\documentclass[10pt]{amsart}

\usepackage{amsmath, amscd, amssymb,xypic, amsfonts, epsfig}
\usepackage[frame,cmtip,arrow,matrix,line,graph,curve]{xy}
\usepackage{graphpap, color, pstricks}
\usepackage[mathscr]{eucal}

\newtheorem{theorem}{Theorem}[section]
\newtheorem{theo}[theorem]{Theorem}
\newtheorem{prop}[theorem]{Proposition}
\newtheorem{coro}[theorem]{Corollary}
\newtheorem{lemm}[theorem]{Lemma}
\newtheorem{conj}[theorem]{Conjecture}
\newtheorem{simp}[theorem]{Simpleness Condition}

\theoremstyle{definition}
\newtheorem{defi}[theorem]{Definition}
\newtheorem{exam}[theorem]{Example}
\newtheorem{rema}[theorem]{Remark}

\usepackage{amsmath, amscd, amssymb}
\usepackage[frame,cmtip,arrow,matrix,line,graph,curve]{xy}
\usepackage{graphpap, color}
\usepackage[mathscr]{eucal}
\usepackage{mathrsfs}
\usepackage[mathscr]{eucal}
\usepackage{pstricks}
\usepackage{color}
\usepackage{cancel}
\usepackage{verbatim}
\usepackage{pstricks}
\usepackage{color}
\usepackage{cancel}

\numberwithin{equation}{section}

\newcommand{\CC}{\mathbb{C}}

\newcommand{\LL}{\mathbb{L}}
\newcommand{\PP}{\mathbb{P}}
\newcommand{\QQ}{\mathbb{Q}}

\newcommand{\ZZ}{\mathbb{Z}}

\def\sO{{\mathscr O}}

\def\sV{\mathscr{V}}

\newcommand{\kk}{\bk}

\newcommand{\cal}{\mathcal}

\def\cA{{\cal A}}

\def\cE{{\cal E}}

\def\cH{{\cal H}}

\def\cM{{\cal M}}
\def\cN{{\cal N}}
\def\cO{{\cal O}}

\def\cV{{\cal V}}
\def\cW{{\cal W}}

\def\cZ{{\cal Z}}
\def\cI{{\cal I}}

\def\fp{\mathfrak{p}}

\newcommand{\tih}{\tilde{h}}

\def\mapright#1{\,\smash{\mathop{\lra}\limits^{#1}}\,}

\let\cN=\cN

\def\dual{^{\vee}}
\def\ucirc{^\circ}
\def\sta{^\ast}

\def\virt{^{\mathrm{vir}}}
\def\upmo{^{-1}}
\def\sta{^{\ast}}

\def\sta{^*}

\def\eps{\epsilon}

\def\lbe{_{\beta}}
\def\lra{\longrightarrow}

\newcommand{\lam}{\lambda}
\newcommand{\si}{\sigma}

\def\begeq{\begin{equation}}
\def\endeq{\end{equation}}
\def\and{\quad{\rm and}\quad}
\def\bl{\bigl(}
\def\br{\bigr)}
\def\defeq{:=}

\def\sub{\subset}
\def\Ao{{\mathbb A}^{\!1}}

\def\Po{{\mathbb P^1}}
\def\and{\quad\text{and}\quad}
\def\mapright#1{\,\smash{\mathop{\lra}\limits^{#1}}\,}

\def\lalp{_\alpha}

 \DeclareMathOperator{\Ext}{Ext}
  \DeclareMathOperator{\Hom}{Hom}
 \DeclareMathOperator{\Aut}{Aut}
 
\DeclareMathOperator{\id}{id} 
 \DeclareMathOperator{\res}{res}

\DeclareMathOperator{\spec}{Spec}

\DeclareMathOperator{\coker}{coker}

\def\bul{^\bullet}

\def\vsp{\vskip5pt}

\def\kk{{\mathbf k}}

\def\PP{\mathbb{P}}
\def\CC{\mathbb{C}}

\def\lra{\longrightarrow}
\def\mapright#1{\,\smash{\mathop{\lra}\limits^{#1}}\,}
\def\mapleft#1{\,\smash{\mathop{\longleftarrow}\limits^{#1}}\,}
\def\cO{\mathcal{O}}

\def\csta{{\CC\sta}}

\def\vdim{\text{vir}.\dim}
\def\sigc{\Sigma^\circ}
\def\DM{DM-}

\def\LL{{\mathbb L}}

\def\gmo{^{\ge -1}}

\def\beq{\begin{equation}}
\def\eeq{\end{equation}}

\title[A wall crossing formula of
DT invariants]{A wall crossing formula of Donaldson-Thomas invariants without Chern-Simons functional}
\author{Young-Hoon Kiem}
\address{Department of Mathematics and Research Institute
of Mathematics, Seoul National University, Seoul 151-747, Korea}
\email{kiem@math.snu.ac.kr}

\author{Jun Li}
\address{Department of Mathematics, Stanford University, Stanford,
USA} \email{jli@math.stanford.edu}
\date{}

\thanks{Young-Hoon Kiem was partially supported by NRF 2010-0007786; Jun Li was partially
supported by NSF grant NSF-0601002.}

\begin{document}
\maketitle

\begin{abstract}
We prove a wall crossing formula of Donaldson-Thomas type invariants without Chern-Simons functionals.
\end{abstract}

\section{Introduction}

A wall crossing formula refers to the change of invariants of moduli
spaces when they undergo birational transformations, like the
variation of moduli spaces of stable sheaves when the stability
condition changes. Wall crossing formulas were investigated
extensively for the Donaldson polynomial invariants of surfaces in
the 1990s by several groups, including Friedman-Qin \cite{FQ},
Ellingsurd-G\"ottsche \cite{EG} and Matsuki-Wentworth \cite{MW}. At
that time, the notion of virtual cycle had not been discovered and
wall crossing formulas were worked out largely for smooth moduli
spaces. Wall crossing formulas involving virtual cycles were taken
up by T. Mochizuki \cite{Moch} in his work on higher rank Donaldson
polynomial invariants of surfaces. All these approaches use
Geometric Invariant Theory (or GIT for short) flips discovered by
Dolgachev-Hu and Thaddeus \cite{Hu, Th}, relying on that the moduli
spaces of sheaves are constructed using GIT \cite{Gie, Maru, Sim}.

For a Calabi-Yau three-fold $S$, the moduli space of stable
sheaves with fixed Chern classes is equipped with a
\emph{symmetric obstruction theory};  the degree of its virtual
fundamental class defines the Donaldson-Thomas invariant of $S$.
In \cite[Thm. 5.9]{Joyce}, using the existence of local \emph{Chern-Simons
functionals}, Joyce-Song prove a wall crossing
formula in case the stability crosses a wall (cf. Definition
\ref{assum1}). Here a local Chern-Simons functional refers to a
function $f$ defined on a smooth $Y$ such that the vanishing of
$df$ defines the germ of the moduli space and its symmetric
obstruction theory. By using Hall algebras and
breaking the moduli spaces into pieces, they define generalized
Donaldson-Thomas invariants for arbitrary Chern classes and find
formulas comparing them.

Recently, the study of moduli of stable sheaves has been extended
to moduli of stable objects in the derived category
$D^b(\mathrm{Coh} S)$ of bounded complexes of coherent sheaves.
In general, the study of moduli of complexes of coherent
sheaves or objects in the derived category cannot be reduced to
the study of sheaves, and the existence of local Chern-Simons
functionals is not known. Therefore, in order to extend the wall
crossing formulas to moduli of stable objects in the derived
category, a new method is required.

In this paper we develop a new method for wall crossing formulas
of Donaldson-Thomas type invariants \emph{without relying on
Chern-Simons functionals}. Our method is not motivic and does not use
Behrend's function. Instead we perform a
\emph{$\CC\sta$-intrinsic blow-up} to resolve the issue of
infinite stabilizers, construct an auxiliary space (called the
\emph{master space}) that captures the wall crossing phenomenon,
and apply the virtual localization formula \cite{GrPand} and our
reduction technique \cite{Kiem-Li}.

We now outline our results. We let $\cM=[X/\CC^*]$ be the quotient
of a separated $\CC\sta$-equivariant Deligne-Mumford (or DM for
short) stack $X$ equipped with an equivariant symmetric
obstruction theory (Definition \ref{def-perobth}).
Suppose $\cM$ contains
two open dense substacks $M_\pm=[X_\pm/\CC^*]$ that are separated
proper \DM stacks. The equivariant symmetric obstruction theory of
$X$ induces symmetric obstruction theories of $M_\pm$. By \cite{BF,
LT}, the symmetric obstruction theories provide their respective
(dimension $0$) virtual cycles $[M_\pm]\virt$. The wall crossing
formula measures the difference $\deg [M_+]\virt-\deg [M_-]\virt$.

Taking the $\CC^*$-fixed part of the symmetric obstruction theory
along the fixed point locus $X^{\CC^*}$, we have an induced
symmetric obstruction theory on $X^{\CC^*}$. We let its zero
dimensional virtual fundamental class be
$$[X^{\CC^*}]\virt = \sum _{k=1}^r a_k [p_k]; \quad p_k\in X^{\CC\sta},\ a_k\in \QQ.
$$

%

%
%
%

\begin{theo}\label{thm1.2} Suppose $M_\pm\subset \cM$ is a
simple flip (Definition \ref{def-master1}) and $\cM=[X/\CC^*]$ has a
symmetric obstruction theory (Definition \ref{def-perobth}). Suppose
further that $X$ can be embedded $\csta$-equivariantly into a smooth
DM-stack (\cite[Appendix C]{GrPand}). Then
\[
\deg [M_+]\virt - \deg [M_-]\virt=\sum_{k} a_k\cdot\lam_k;\quad
 \lam_k=(-1)^{n_k-1}\sum_jn_{k,j}/j ,
\]
where $n_{k,j}$ is the dimension of the weight $j$ part of the Zariski tangent space $T_{X,p_k}$ and $n_k=\sum_j n_{k,j}$.

In case $\lam_k=\lambda$ is independent of
$k$, then
\[
\deg [M_+]\virt - \deg [M_-]\virt=\lambda\cdot \deg
[X^{\CC^*}]\virt.
\]
\end{theo}

We comment that the proof does not require the existence of local
Chern-Simons functionals.

Our proof goes as follows. We first construct a
$\CC\sta$-intrinsic blow-up $\bar X$ of $X$ along $X^{\CC^*}$
and show that the quotient $[\bar X/\CC\sta]$ has an induced perfect
obstruction theory. We then construct a master space $\bar Z$ for
$\bar X$, analogous to that in \cite{Th} and apply the localization
by cosection technique \cite{Kiem-Li} to construct a reduced virtual
fundamental class of $\bar Z$. Applying the virtual localization
formula of \cite{GrPand} to $\bar Z$, we obtain Theorem
\ref{thm1.2}.
\vsp

It is worthwhile to comment on the assumptions of Theorem
\ref{thm1.2}. As shown below, the existence of $X$ is assured in
quite a general setting. But embedding $X$ $\csta$-equivariantly
into a smooth DM-stack seems restrictive. In two occasions of the
proof we use this assumption; one is constructing the obstruction
theory of the $\csta$-intrinsic blow-up, the other is applying
Graber-Pandharipande's virtual $\csta$-localization theorem.

Here is what we plan to do to remove this technical assumption in
the near future. Firstly, we intend to develop a {\sl local
obstruction theory} that consists of an \'etale atlas $U\lalp\to
X$, obstruction theories $\eta\lalp: F\lalp\bul\to\LL_{U\lalp}$, and
compatibility of the obstruction theories 
over $U_{\alpha\beta}=U\lalp\times_X U\lbe$. Such a theory is known
to the experts, and the usual tools on virtual cycles apply. This
seems necessary since it is hard to construct a global universal family of
the moduli stack of derived category objects.
Secondly, we intend to prove the
$\csta$-localization theorem in the setting of local
obstruction theory. This seems to require more efforts, but should
be achievable in the near future.

In light of this, in this paper we also prove some technical results
aimed at proving the following generalization.

\begin{conj}\label{conj} Theorem \ref{thm1.2} holds true without assuming
the existence of a $\csta$-equivariant embedding of $X$ into a
smooth DM-stack.
\end{conj}

We now comment on how Theorem \ref{thm1.2} applies toward a
wall crossing formula of Donaldson-Thomas invariants of moduli of
derived category objects over a Calabi-Yau three-fold $S$. Let $\tau_\pm$
be two stability conditions crossing a simple wall $\tau_0$
(Definition \ref{assum1}); let $\cM$ be the moduli stack of
$\tau_0$-semistable objects. We construct a separated \DM stack
$X$ parameterizing pairs $(E,\si)$, where $E\in\cM$ and
$\sigma:L(E)\mapright{\cong} \CC$ for some 1-dimensional vector
space $L(E)$ associated to $E$. Scaling $\si$ with elements in
$\CC\sta$ makes $X$ a $\CC\sta$-equivariant stack. Then
$\cM=[X/\CC^*]$, while the moduli stacks of $\tau_\pm$-stable
objects are open substacks $M_\pm=[X_\pm/\CC^*]$ for two open
substacks $X_\pm\sub X$ (cf. Proposition \ref{p.2.10}). 
Because Conjecture 1.2 is not proved yet, we suppose that $X$
equivariantly embeds into a smooth \DM stack. This condition is
always satisfied when $\cM$ is constructed as a GIT quotient. In
particular, our results apply directly to the moduli of
\emph{sheaves} or \emph{stable pairs}. We show in Lemma \ref{p2.8}
that $X$ has a $\CC^*$-equivariant symmetric obstruction theory.
It is now immediate to deduce the following wall crossing formula
from Theorem \ref{thm1.2}.

\begin{coro}\label{coro1.3} 
Let $\tau_\pm$
be two stability conditions crossing a simple wall $\tau_0$
(Definition \ref{assum1}); let $\cM$ be the moduli stack of
$\tau_0$-semistable objects $D^b(\mathrm{Coh}S)$ with the same fixed Chern classes, and let 
$M_\pm\sub\cM $ be the moduli stacks of $\tau_\pm$-stable objects. Suppose 
$M_\pm\sub\cM$ is a simple wall crossing of each other (cf. Definition \ref{assum1}).
Suppose the DM stack $X$ constructed in Proposition \ref{p.2.10} embeds $\csta$-equivariantly
in a smooth DM stack. Suppose further that $Ext^{\le
-1}(E,E)=0$ for all $E\in \cM$. Then 
$$\deg [M_+]\virt - \deg [M_-]\virt =
(-1)^{\chi(E_1,E_2)-1}\cdot \chi(E_1,E_2)\cdot \deg [{M_1}] \virt \cdot \deg
[{M_2}]\virt,
$$
where $E_i\in M_i$, $M_i$ are the moduli stacks 
specified in Definition \ref{assum1} and $\chi(E_1,E_2)=\sum(-1)^i\dim \Ext^i(E_1,E_2)$.

\end{coro}

In \S2, we show that a simple wall crossing $M_\pm\subset \cM$ in
$D^b(\mathrm{Coh} S)$ is described as a simple flip
$[X_\pm/\CC^*]\subset [X/\CC^*]$. In \S3, we construct an intrinsic
blow-up of $\CC^*$-quotients and study the induced perfect
obstruction theory. In \S4, we construct the master space and in
\S5, we apply the virtual localization formula of \cite{GrPand} to
obtain Theorem \ref{thm1.2} and Corollary \ref{coro1.3} by careful
local calculations. In the Appendix, we discuss an analogous wall
crossing result for non-symmetric obstruction theories.

\black

\section{Donaldson-Thomas invariants and wall crossing }

In this section, we explain how a simple wall crossing (Definition
\ref{assum1}) can be described by a simple flip (Definition
\ref{def-master1}) in a global $\CC^*$ quotient equipped with an
equivariant obstruction theory. In later sections, we will work
out a wall crossing formula for simple flips.


\subsection{Simple wall crossing}\label{subsection2.1}

In this paper, we fix a smooth projective Calabi-Yau three-fold $S$.
Let $\tau_0$ be a stability condition in some abelian subcategory
$\cA_0$ of the derived category $D^b(\mathrm{Coh}(S))$ of bounded
complexes of coherent sheaves. Let $\cM$ be the moduli stack of
$\tau_0$-semistable objects, with fixed Chern classes. By tensoring
with some power  of $\cO_S(1)$ if necessary, we may assume
$\chi(E)\ne 0$ for $E\in \cM$.

Let $\tau_\pm$ be nearby stability conditions of $\tau_0$ and let
$M_\pm$ be the moduli stack of $\tau_\pm$-stable objects. In this
paper, we consider the following case of \emph{simple} wall
crossing.

\begin{defi}\label{assum1} We say $M_+$ is obtained from $M_-$ by a
\emph{simple wall crossing} in $\cM$ if the following conditions
are satisfied:
\begin{enumerate}
\item Strictly $\tau_0$-semistable objects $E$ (i.e. semistable
but not stable) in $\cM$ have $\mathrm{gr}(E)=E_1\oplus E_2$ with
$E_1\in M_1$ and $E_2\in M_2$, where $M_i$ are the moduli stacks of $\tau_0$-stable objects
of Chern classes equal to $c(E_i)$. Further, both $M_1$ and $M_2$ are proper
separated \DM  stacks;
\item the factors
$E_i\in M_i$ in (1) satisfy $\chi(E_1)=\chi(E_2)$; there is a
$B\in K(S)$ such that
$\chi(E_1\otimes B)\ne \chi(E_2\otimes B)$;
\item $\tau_+$-stable (resp. $\tau_-$-stable) objects are either
$\tau_0$-stable or non-split extensions of $E_2$ by $E_1$ (resp.
$E_1$ by $E_2$) for some strictly $\tau_0$-semistable objects
$E_1\oplus E_2$ with $E_1\in {M_1}$ and $E_2\in {M_2}$;
\item $M_\pm$ are proper separated \DM stacks.
\end{enumerate}
\end{defi}

It is immediate that $M_\pm$ are open substacks of $\cM$. The wall crossing formula
compares the degrees of the virtual fundamental classes of $M_+$ and $M_-$.

We remark that part of the Definition requires that all elements in $M_1$ (resp. $M_2$)  have identical
Chern classes $c(E_1)$ (resp. $c(E_2)$). The case where the wall crossing occurs at different
$c(E_1)$ and $c(E_2)$ can be treated the same way since then $M_1$ and $M_2$ splits to pairs
$(M_{1,j},M_{2,j})$, where elements in $M_{1,j}$ and $M_{2,j}$ have identical Chern classes.


\subsection{Simple flip}\label{subsection2.2}
\def\ti{\tilde}

In this subsection, we introduce our geometric set-up for simple
wall crossing. Let $X$ be a separated Deligne-Mumford stack of
finite type acted on by $\CC^*$. Let $\cM=[X/\CC^*]$ be the
quotient stack.

Let $X^{\CC^*}$ denote the fixed point locus in $X$. We let
$X_0\sub X-X^{\CC^*}$ be the open substack of $x\in X$ so that the
orbit ${\CC^*}\cdot x$ is 1-dimensional and closed in $X$, and let
$$\sigc_\pm=\{x\in X-X_0\cup X^{\CC^*}\mid \lim_{\sigma\to 0} {\sigma^{\pm 1}}\cdot x\in
X^{\CC^*}\}.
$$
We then form
\beq\label{de1.2}\Sigma_{\pm}=\sigc_\pm\cup
X^{\CC^*}, \eeq which are $\CC\sta$-invariant, and let
\beq\label{de1.3} X_\pm=X-\Sigma_{\mp}\subset X, \qquad
M_\pm=[X_\pm/{\CC^*}]\subset\cM=[X/\CC\sta]. \eeq

\begin{defi}\label{def-master1}
We say that the pair $M_\pm=[X_\pm/\CC\sta]\sub \cM$ is a
\emph{simple flip} if the following hold:
\begin{enumerate}
\item[(i)] The fixed locus $X^{\CC\sta}$ is a proper \DM stack;
\item[(ii)] $\Sigma_\pm$ are closed in $X$;
\item[(iii)]  both $M_\pm=[X_\pm/\CC\sta]$ are proper separated \DM
stacks;
\item [(iv)] the pair $X_\pm\sub X$ satisfies the Simpleness Condition stated below.
\end{enumerate}
\end{defi}

We first fix our convention on discrete valuation rings and their finite extensions.
Let $R$ be a discrete valuation ring. We always denote by $K$ its field of fractions;
$\xi$ and $\xi_0$ be the generic and closed points of $\spec R$. When $\ti R\supset R$ is a finite
extension, we denote by $\ti K$, $\ti \xi$ and $\ti \xi_0$ the corresponding field of fractions, the generic point
and the closed point. We let $\zeta$ and $\ti\zeta$ be uniformizing parameters of $R$ and $\ti R$, respectively.
We denote by $\imath(\ti R)$ the ramification index of $\ti R\supset R$; namely, $\zeta=\ti\gamma\cdot \ti\zeta^{\imath(\ti R)}$ 
for an intertible $\ti\gamma\in \ti R$.
In case $f:\spec R\to X$ is a morphism, we denote by $\ti f:\spec \ti R\to X$ the composite $\spec \ti R\to \spec R\to X$.
Given $g: \spec K\to \csta$, we denote by $g\cdot f: \spec K\to X$ the morphism induced by
the $\csta$-action on $X$.

\begin{simp}\label{simpcond} Let $R$ be a discrete valuation ring over $\CC$.
Let $f_+:\spec R\to X$; let $g: \spec K\to
{\csta}$ be of the form $g\sta(t)=\alpha\cdot
\zeta^a$ for an invertible $\alpha\in R$ and $a>0$.
Then the following hold:
\begin{enumerate}
\item[(a)]  Suppose $f_+(\xi)\in X_0$,
$f_+(\xi_0)\in \Sigma_+^\circ$, and $g\cdot
f_+: \spec K\to X_0$ extends to $f_-: \spec R\to X$ so that $f_-(\xi_0)\in\Sigma_-^\circ$.
Then there is a finite extension $\tilde R\supset R$ such that for any
invertible $\ti\beta\in\ti R$ and any integer $0< b< \imath(\ti R)\cdot a$, a morphism $\ti h:\spec \ti K\to {\csta}$ of the form
$\ti h^{\ast}(t)=\ti\beta\cdot \ti\zeta^{b}$ has the property  that
$\ti h\cdot \ti f_+:\spec \tilde K\to X$ extends to $(\ti h\cdot \ti f_+)^\sim: \spec \tilde R\to X$ with 
$(\tih\cdot \ti f_+)^\sim(\tilde \xi_0)\in X^{\csta}$;
\item[(b)]  Suppose $f_+: \spec R\to \Sigma_+^\circ$. Then possibly after a finite extension $\tilde
R\supset R$, $\ti g\cdot \ti f_+: \spec \tilde K\to \Sigma_+^\circ$ extends to $(\ti g\cdot \ti f_+)^\sim: \spec \tilde R\to \Sigma_+$ 
such that $(\ti g\cdot \ti f_+)^\sim(\tilde \xi_0)\in X^{\csta}$.
\item[(c)]   Both (1) and (2) hold with``$+$" replaced by``$-$", $a>0$ replaced by $a<0$, and $0<b< \imath(\ti R)\cdot a$ replaced by 
$\imath(\ti R)\cdot a< b<0$.
\end{enumerate}
\end{simp}

\begin{exam}
Let $V_+,V_-,V_0$ be vector spaces on which $\CC^*$ acts with
weights $1$, $-1$, and $0$ respectively. Let $V=V_+\times V_0\times V_-$
and $X=\PP V-(\PP V_+\cup \PP V_-)$. Then $\Sigma_\pm=V_\pm\otimes
\cO_{\PP V_0}(1)$ and $X_\pm=\PP V-\PP (V_0\times V_\mp)$. Hence
$M_+=\PP V_+\times \PP (V_0\times V_-)$ and $M_-=\PP V_-\times \PP
(V_0\times V_+)$.

Suppose in addition that $M_\pm$ have symmetric obstruction theories.
Since $M_\pm$ are smooth, their virtual cycles 
$[M_\pm]\virt$ are equal to the Euler classes $e(\Omega_{M_\pm})$. Therefore the wall crossing formula is
\[ \deg[M_+]\virt-\deg[M_-]\virt=
(-1)^{n_+-n_--1}(n_+-n_-)\cdot \deg[X^{\CC^*}]\virt,
\]
where $n_\pm=\dim V_\pm$ is the dimension of the $\pm 1$-weight
space in the normal space to $X^{\CC^*}$ in $X$.
\end{exam}

\begin{exam}\label{ex2}(\cite{Simp})
More generally, suppose there is a ${\CC\sta}$-equivariant
separated proper scheme $W$ such that
$$X=W-\{x\in W \,|\, \lim_{t\to\infty}t\cdot x\in F^+\}-
\{x\in W \,|\, \lim_{t\to 0}t\cdot x\in F^-\}
$$
where $F^\pm$ are parts of a partition $W^{\CC\sta}=F^+\sqcup
F^-\sqcup F^0$ of the fixed point set $W^{\CC^*}$ of $W$,
satisfying 
\begin{enumerate}
\item if $u\preceq v$ and $v\in F^+$, $u\in F^+$;
\item if $u\preceq v$ and $u\in F^-$, then $v\in F^-$;
\item if $u\precneqq v$ and $v\in F^0$, then $u\in F^+$;
\item if $u\precneqq v$ and $u\in F^0$, then $v\in F^-$,
\end{enumerate}
with respect to the ordering $\preceq$ defined as follows: For
$x,y\in W^{\CC\sta}$, $x\preceq y$ if there exists a sequence
$x_1,\cdots, x_n\in W$ such that
$$\lim_{t\to 0}t\cdot x_1=x,\ \lim_{t\to \infty}t\cdot x_n=y,\
\lim_{t\to \infty}t\cdot x_k=\lim_{t\to 0}t\cdot x_{k+1}, \quad
\forall k.
$$
Then $M_\pm=[X_\pm/\CC^*]$ is a simple flip. The proof of the simpleness condition is essentially contained in \cite[Theorem 11.1]{Simp}. Notice that all GIT $\CC^*$-flips are special cases of this example. 
\end{exam}

\subsection{Simple wall crossings and simple flips}\label{subsection2.3}

In this subsection, we investigate a case where a simple wall
crossing of moduli of derived category objects is a simple flip.

Let $M_\pm\sub \cM$ be a moduli of semistable derived category objects as in
Definition \ref{assum1}. We first show that $\cM$ can be written as
a global $\CC^*$ quotient of a separated \DM stack $X$. Fix $B\in
K(S)$ that distinguishes elements in $M_1$ and $M_2$, as stated in Definition
\ref{assum1}.

\begin{defi} \label{def-det}
For $E\in \cM$, let
$$L(E):=\det R\Gamma(E\otimes B)^{\chi(E)}\cdot
\det R\Gamma(E)^{-\chi(E\otimes B)},
$$
where $\det
R\Gamma(E)=\otimes_i ( \wedge^{top}H^i(E))^{(-1)^i}$.
\end{defi}

Observe that if $E=E_1\oplus E_2$ is strictly $\tau_0$-semistable,
then any $\varphi\in\mathrm{Aut}(E)/\CC^*$ ($\cong \CC\sta$), where
$\CC\sta\le\mathrm{Aut}(E)$ is the subgroup $\CC\sta
\cdot\mathrm{id}_E$, has a representative $\lambda\cdot
\mathrm{id}_{E_1}\oplus \lambda^{-1}\cdot \mathrm{id}_{E_2}$. The
weight of this circle action on $L(E)$ is
\begin{equation}\label{defnu} \nu:=(\chi(E_1\otimes
B)-\chi(E_2\otimes B))\cdot \chi(E)\ne 0.
\end{equation} 
Note that by assumption, $\chi(E)=2\chi(E_1)=2\chi(E_2)\ne 0$.

Note that the ordering of $M_1$ and $M_2$ is determined by the choice of $\tau_+$ and $\tau_-$
according to Definition \ref{assum1}. Thus by
interchanging $\tau_+$ and $\tau_-$ if necessary, we may and will assume $\nu>0$.

\begin{defi}
Two pairs $(E_1,\sigma_1)$ and $(E_2,\sigma_2)$ where
$\sigma_i:L(E_i)\mapright{{}_{\cong}} \CC$, are \emph{isomorphic} if
there exists a $\varphi:E_1\mapright{{}_{\cong}} E_2$ such that
$\sigma_2\circ L(\varphi)=\sigma_1$.
\end{defi}

The automorphism group $\mathrm{Aut}(E,\sigma)$ of such a pair
$(E,\sigma)$ is the group of all isomorphisms from $(E,\sigma)$ to
itself. Obviously, $\CC\sta\cdot \id_E\le \Aut(E,\sigma)$.

\begin{lemm}\label{split}
$\mathrm{Aut}(E,\sigma)/\CC^*\cdot \mathrm{id}_E$ is a finite group.
\end{lemm}

\begin{proof}
Unless $E=E_1\oplus E_2$, the quotient group is trivial. When
$E=E_1\oplus E_2$, any $\lambda\cdot \mathrm{id}_{E_1}\oplus
\lambda^{-1}\cdot \mathrm{id}_{E_2}$
should satisfy
$\lambda^{\nu}=1$
with $\nu$ defined in \eqref{defnu}. Hence the group is finite.
\end{proof}


\begin{prop}\label{p.2.10} Let $M_\pm\sub \cM$ be as in Definition
\ref{assum1} and suppose for $E\in \cM$ $Ext^{\le -1}(E,E)=0$.
\begin{enumerate}
\item[(1)] The groupoid $X$ of families of pairs $(E,\sigma)$ of $E\in
\cM$ and $\sigma:L(E)\cong \CC$ is a separated \DM stack with a
strict representable\footnote{The strictness is defined in \cite{C-action}; 
the $\csta$-action is representable if the multiplication morphism
$\csta\times X\to X$ is representable.}  $\CC\sta$-action by $t\cdot (E,\sigma)=(E,t\sigma)$; the fixed
point locus is $M_1\times M_2 = \{(E_1\oplus
E_2,\sigma)\,|\,E_i\in
M_i\}/\cong$.
\item[(2)] $\cM\cong [X/\CC^*]$; let $X_\pm\sub X$ be as defined in \eqref{de1.3},
then $M_\pm\cong [X_\pm/\CC^*]$.
\item[(3)] The pair $M_\pm\cong [X_\pm/\CC^*]$ is a simple flip
(Definition \ref{def-master1}).
\end{enumerate}
\end{prop}

\begin{proof}
By \cite{Lieb}, $\cM$ is an Artin stack locally of finite
presentation. The relative version of the construction of $L(E)$
from $E$ shows that $X$ is an Artin stack. Since every
$(E,\sigma)\in X$ has finite stabilizer, $X$ is actually a \DM
stack by slice argument.
The strictness and the representable property of the $\csta$-action
follows directly from the construction.

It follows from the property of stability and from Definition \ref{assum1} that
$X^\csta$ consist of $E_1\oplus E_2$, $E_i\in M_i$. We now prove that $(E,\si)\in X_\pm$
if and only if $E$ is $\tau_\pm$-stable.
For this, we need a description of $\Sigma_\pm^\circ$.

Let $E\in M_+$ be $\tau_+$-stable but not $\tau_0$-stable. By Definition \ref{assum1}, $E$ is a
non-trivial extension of an $E_2\in M_2$ by an $E_1\in M_1$. Let $\alpha\in \Ext^1(E_2,E_1)$ be the
extension class of the distinguished triangle $E_1\to E\to E_2\to$. Let $0\in \Ao$ be the origin; let $e\in\Gamma(\cO_{\Ao}(0))$
be the image of $1\in\Gamma(\cO_{\Ao})\sub \Gamma(\cO_{\Ao}(0))$. Then $e$ vanishes simply at $0\in {\Ao}$.
Let $p_S$ and $p_{\Ao}$ be the projections of $S\times {\Ao}$. Then 
$$\alpha\cdot e\in\Ext^1_{S\times {\Ao}}(p_S\sta E_2\otimes p_{\Ao}\sta\cO_{\Ao}(-0),p_S\sta E_1)=
\Ext^1(E_2,E_1)\otimes\Gamma(\cO_{\Ao}(0))
$$
defines an $\cE\in D^b(\text{Coh}(S\times {\Ao}))$ that fits into the follwoing diagram of distinguished triangles
\beq\label{dt-C}
\begin{CD}
p_S\sta E_1@>>> \cE@>>> p_S\sta E_2\otimes p_{\Ao}\sta\cO_{\Ao}(-0)@>+1>>\\
@| @VVV @VVV\\
p_S\sta E_1 @>>> p_S\sta E @>>> p_S\sta E_2 @>+1>>\\
@. @VVV @VVV\\
@. \iota_0\sta E_2@=  \iota_0\sta E_2\\
@.@V+1VV @V+1VV\\
\end{CD}
\eeq
where $\iota_0: S\times 0\lra S\times {\Ao}$ is the central fiber.

We now form a line bundle on ${\Ao}$
$$L(\cE)=\det Rp_{{\Ao}\ast}(\cE\otimes p_S\sta B)^{\chi(E)}\otimes 
\det Rp_{{\Ao}\ast}(\cE)^{-\chi(E\otimes B)}.
$$
Using \eqref{dt-C}, we obtain canonical isomorphisms
$$L(\cE)\otimes L(\iota_0\sta E_2)\cong L(p_S\sta E)\cong L(E)\otimes\cO_{\Ao}.
$$
Since canonically
$L(\iota_0\sta E_2)\cong \cO_{\Ao}(-\nu' 0)$, where $\nu'=\frac12\nu$ and $\nu>0$ is as in \eqref{defnu},
we have
\beq\label{iso-det}
L(\cE)\cong L(E)\otimes\cO_{\Ao}(\nu' 0).
\eeq

Now let $\si: L(E)\to\CC$ be a fixed isomorphism; we fix an isomorphism $\ti\si: L(\cE)\cong\cO_{\Ao}$.
Then restricting the isomorphism \eqref{iso-det} to ${\Ao}-0$, for a $u\in\Gamma(\cO_{\Ao})$ that vanishes simply at $0\in {\Ao}$
we have
$$\ti\si= u^{\nu'}\cdot\sigma.
$$
Since $\cE|_{S\times({\Ao}-0)}\cong p_S\sta E|_{S\times({\Ao}-0)}$ and $\iota_0\sta \cE=E_1\oplus E_2$,
this proves
$\lim_{t\to 0} t\cdot (E,\si)=(E_1\oplus E_2, \ti \si|_0)\in X^\csta$. Hence $\tau_+$-stable but not $\tau_0$-stable objects
are all over $\Sigma_+\ucirc$.

For the other direction, suppose $(E,\si)\in X-X^\csta$ such that 
\beq\label{lim}
\lim_{t\to 0} t\cdot (E,\si)=(E_1\oplus E_2, \cdot) \in X^\csta.
\eeq
Let $R=\CC[t]_{(t)}$, localized at the ideal $(t)$. 
We let $(\cE,\ti\si)$ on $S\times \spec R$ be given by the morphism $\spec R\to X$ that is
given by the completion of the morphism $\csta\cong \{t\cdot (E,\si)\mid t\in\csta\}\to X$. Here the
completion exists because of \eqref{lim}. As usual, we 
denote by $\xi$ and $\xi_0$ the generic and the closed points of $\spec R$. 
By the construction of $\cE$, we have isomorphism
\beq\label{iso-A}
\varphi_\xi: \cE|_{S\times \xi}\mapright{\cong} p_S\sta E|_{S\times \xi}\quad
\text{where}\ p_S: S\times\spec R\to S.
\eeq
Thus $\Hom_{S\times\spec R}(p_S\sta E,\cE)$ is a non-trivial $R$-module. 
Let $\varphi\in \Hom_{S\times\spec R}(p_S\sta E,\cE)$ be such that $\varphi|_{S\times\xi}=\varphi_\xi$.
By replacing $\varphi$ by $t^c\varphi$ for some $c\le 0$, we can assume that $\varphi|_{S\times\xi_0}$ is
non-trivial. Namely, for the
closed embedding $\iota: S\times\xi_0\to S\times \spec R$, 
$\iota\sta\varphi: \iota\sta\cE=E_1\oplus E_2\to E$ is non-trivial.

Since $(E,\si)\notin X^\csta$, since there is a non-trivial $E_1\oplus E_2\to E$ and since both $E_i$ are
$\tau_0$-stable and $E$ is $\tau_0$-semistable, $E$ is either an extension of $E_2$ 
by $E_1$ or vice versa. By a parallel discussion as
above, $E$ must be an extension of $E_2$ by $E_1$. This proves that $X_-=X-\Sigma_+$.
The same holds for $X_+$. This proves (2).

Finally, it is direct to check that (3) holds for $X$, knowing that $M_\pm$ are proper and separated.

We prove the simpleness condition in the next Lemma.
\end{proof}

\begin{lemm}\label{lem-auxprop}
The pair $X_\pm\sub X$ satisfies the Simpleness Condition.
\end{lemm}

\begin{proof}
We let $f_\pm$ and $g$ be as in the statement of Simpleness Condition \ref{simpcond}.
We let $z_0=\lim_{t\to 0} t\cdot f_+(\xi_0)$, $z_0\in
X^{{{\csta}}}$ by our construction of $\Sigma_+^\circ$. We pick an
affine \'etale $\rho: U\to X$ that covers
$z_0$. We let $\overline U$ be a compactification of $U$.

We next let $p_1$ and $p_2$ be the first and the second
projections of $\spec R\times{{\csta}}$. We form the morphism
$$F'=p_2\cdot (f_+\circ p_1):\spec R\times \csta\to X,
$$ 
where $\cdot$ denotes the group
action. Since the closure of the image of $F'$ in $X$ contains
$z_0$, there is an open $V'\sub\spec R\times{{\csta}}$ such that
$(\xi_0,c)\in V'$ for general closed $c\in \csta$ and 
$F'(V')$ is contained in the image $\rho(U)\sub X$. 

Thus by replacing $R$ by
a finite extension of $R$, which we still denote by $R$,
we can assume that there is an open $V\sub \spec R\times{{\csta}}$ so that $(\xi_0,c)\in V$ 
for general closed $c\in \csta$; $F'$ lifts to 
an $F: V\to U$ so that we have a commutative diagram
\beq\label{diag-S}
\begin{CD}
V  @>{ F}>> U\\
@VVV@VV{\rho}V\\
\spec R\times \csta @>{F'}>> X,
\end{CD}
\eeq
and $\lim_{t\to 0}  F(( \xi_0, t))=\bar{z_0}\in U$ with $\rho(\bar {z_0})={z_0}$.

We remark that with the new $R$, the induced morphisms $f_\pm:\spec R\to X$ satisfy the
same property, except that the integer $a$ is replaced by its multiple by the ramification index of
the finite extension. Thus the simpleness condition is invariant under finite extensions
of $R$.

We then let $S=\spec R\times\Po $, which contains $\spec 
R\times{{\csta}}$ via ${{\csta}}=\Ao-\{0\}\sub \Ao\sub\Po $. Since $S$ is a
smooth surface, there is a minimal finite set $A_0\sub S$ so that
$ F$ extends to
$$F_0: S\setminus A_0\lra \overline U.
$$
(We call $A_0$ the indeterminacy locus of $F_0$.) We let ${{\csta}}$
act on $S$ by acting trivially on $\spec R$ and tautologically
on ${{\csta}}$ (with weight 1).

We now partially resolve the indeterminacy of $F_0$. We let
$A_0'=A_0\cap S^{{{\csta}}}$. We blow up $S$ along $A_0'$, obtaining
$S_1$, which has an induced ${{\csta}}$-action. We let $A_1\sub S_1$
be the indeterminacy locus of the extension $F_1: S_1\setminus
A_1\to \overline U$ of $F_0$. We then let $A_1'=A_1\cap S_1^{{{\csta}}}$,
blow up $S_1$ along $A_1'$ to obtain $S_2$, etc. We continue this
process until we get $S_k\to S_{k-1}$ so that the indeterminacy
$A_k$ of the extension
$$F_k: S_k\setminus A_k\to \overline U
$$ 
has no point fixed by ${{\csta}}$.

We let $c\in \csta$ be a general closed point and consider the morphism
$\varphi_1: \spec R\to S$ defined via $\varphi_1(\xi)=(\xi,c)$. We let
$\eta_1\defeq \varphi_1(\xi_0)=(\xi_0,c)\in S$.
Since $(\xi_0,c)\in V$ by our choice of $V$, we can assume
$\eta_1\in S\setminus A_0$. Also since $f_+(\xi_0)\in \Sigma_+^\circ$, which is not
fixed by $\csta$, $\eta_1\in S$
is not fixed by $\csta$. Since $S_k$ is the successive blow-up of $A_i^\csta$, 
we can view $\eta_1$ as an element in $S_k$. Thus  $\eta_1\in S_k-S_k^\csta\cup A_k$.

We next let $\iota: \spec K\to \spec  R$ be the inclusion, and form
$(\iota, g): \spec K\to \spec R\times{{\csta}}$, where $g$ is given in the statement of the Lemma.
We let $\varphi_-:\spec R\to S_k$ be the extension of $(\iota,g)$. Since
$\ti p_1: S_k\to\spec R$ induced by the first projection $p_1:S\to\spec R$ is proper,
$\varphi_-$ exists. We let $\eta_-=\varphi_-(\xi_0)\in S_k$.

We now find a chain of $\csta$-invariant rational curves that connects $\eta_1$ and $\eta_-$ in
$S_k$.
Let $\pi: S_k\to S$ be the blow-up morphism. Let
$D=\pi\upmo(\xi_0\times\Po)$; $D$ is connected and is a union of
${{\csta}}$-invariant rational curves. 
As we have argued, $\eta_1$ does not lie in the
exceptional divisor of $\pi$, thus $\eta_1$ lies
in a unique $B_1\cong \Po\sub D$. Since $D$ is connected and since
$\eta_-\in D$, we can find a chain of rational curves $B_1,\cdots,
B_l$ so that $q_i=B_i\cap B_{i+1}\ne\emptyset$ and $\eta_-\in B_l$.

It is easy to describe the group action on $B_i$. Let $\eta_i$, $i\ge 2$, be a closed point
in $B_i-B_i^\csta$. Since $\eta_1$ is not fixed by $\csta$,
we can find $\eps\in\{1,-1\}$ so that $\lim_{t\to 0} t^\eps\cdot\eta_1=q_1$.
Since the $\csta$-action on $\spec R\times \Po$ is via weight $0$ and $1$ on the two factors,
since $S_k$ is derived by successively blowing up a collection of $\csta$-fixed points, and since
$B_i$ is a chain of $\csta$-invariant rational curves, we must have $\lim_{t\to 0} t^\eps\cdot \eta_i=q_i$ for
$i\le l$. (We let $q_l$ be such that $B_l^\csta=\{q_{l-1},q_l\}$.) 
Because $g(\xi)\cdot f_+(\xi)=f_-(\xi)$ with $g\sta(t)=\alpha\cdot \zeta^a$ and
$a>0$, we have $\eps>0$. This proves $\eps=1$.

We claim that $F_k|_{B_1}\ne \text{const}$. Because $f_+( \xi_0)\in\Sigma_-^\circ$, and
because $\rho(F_k(\eta_1))=f_+(\xi_0)$, $\rho(F_k(\eta_1))\in\Sigma_+^\circ$.
Since $F_k|_{B_1}$ is $\csta$-equivariant, $F_k|_{B_1}\ne\text{const}$. This proves the claim.

We let $\ti V=F_k\upmo(U)\sub S_k$. By \eqref{diag-S}, $q_1\in V\sub \ti V$. 
In general, if for some $2\le i\le l$ we know
$q_{i-1}\in {\ti V}$, then $B_{i}\cap {\ti V}\ne \emptyset$; then we can talk about whether $F_k|_{B_{i}}$ is constant;
if in addition $F_k|_{B_{i}}$ is constant, then $q_{i}\in {\ti V}$. Thus by an induction argument,
we see that there is an $2\le m\le l+1$ so that (1) $F_k|_{B_i}$ is constant for all $2\le i< m$, 
and $F_k|_{B_{m}}$ is not constant in case $m\ne l+1$; 
and (2) $B_i\sub \ti V$ for $2\le i<m$ and $B_m\cap\ti V\ne \emptyset$.

We now prove item (a) of the simpleness condition. Let $m$ be the integer specified. We claim that in case (a), $m\le l$.
Suppose not, then $F_k(B_i)=F_k(q_1)\sub U$ for all $2\le i\le l$.
Since $\rho$ is \'etale, this implies $\rho(F_k(B_i))\sub X^\csta$ for the same $i$; hence $\rho(F_k(\eta_-))\in X^\csta$. 
But this contradicts $\rho(F_k(\eta_-))=f_-(\xi_0)$ and the assumption $f_-(\xi_0)\in \Sigma_-^\circ$.
Therefore, $m\le l$ and consequently, $F_k|_{B_m}$ is not constant.
Since $B_m\cap \ti V\ne \emptyset$, we may let $\eta_m\in B_m\cap\ti V$ be a general closed point.

We claim that $\rho(F_k(\eta_m))\in\Sigma_-^\circ$. 
By the definition of $\Sigma_-\ucirc$, $\rho(F_k(\eta_m))\in\Sigma_-^\circ$ if
$\rho(F_k(\eta_m))\notin X^\csta$ and $\lim_{t\to 0} t\upmo\cdot 
\rho(F_k(\eta_m))\in X^\csta$. 
Suppose $\rho(F_k(\eta_m))\in X^\csta$. As argued before, $\rho\circ F_k$ is $\csta$-equivariant 
(where it is defined), and $\rho$ \'etale. Thus $\rho(F_k(\eta_m))\in X^\csta$ implies that $F_k|_{B_m}$ is constant, a contradiction. 
For the remaining condition, we notice that
$$\lim_{t\to 0} t\upmo\cdot \rho(F_k(\eta_m))=\lim_{t\to 0} \rho(F_k(t\upmo\cdot \eta_m))=\rho(F_k(q_{m-1}))
={z_0}\in X^\csta.
$$
This proves that $\rho(F_k(\eta_m))\in\Sigma_-\ucirc$.

Then since both $[f_-(\xi_0)]$ and
$[\rho(F_k(\eta_{m} ))]\in [X_-/{{\csta}}]$ are specializations of
$[f_-(\xi)]=[f_+(\xi)] \in [X_-/{{\csta}}]$, by the
assumption that $[X_-/{{\csta}}]$ is separated, we conclude that
\beq\label{ff}
[f_-(\xi_0)]=[\rho(F_k(\eta_{m} ))]\in [X_-/\csta].
\eeq

We prove $l=m$. For this, we use the properties of the construction of $X$. Suppose $l>m$.
Since $S_k$ is the result of successive blow-ups
of $\csta$-fixed points of $S=\spec R\times\Po$,
for the $\eta_m\in B_m$, there are a finite extension $\ti R\supset R$ and a morphism 
$\ti e: \spec\ti R\to \csta$ so that
$\ti e\cdot \ti\varphi_-$ extends to $\ti\varphi_m:\spec \ti R\to S_k$ so that $\ti\varphi_m(\ti\xi_0)=\eta_m$.
By the property of the $\csta$-actions on the chain $B_i$, we have $\ti e\sta(t)=\ti\beta\cdot\ti\zeta^r$ with $r<0$ and
$\ti\beta\in\ti R$ invertible. 

Since $\eta_m\in {\ti V}$, $\rho\circ F_k\circ \ti\varphi_m: \spec \ti R\to X$ is defined. Let 
$$\ti f_m\defeq \rho\circ F_k\circ\ti \varphi_m :\spec \ti R\lra X.
$$
(Let $\ti f_-$ be the composite of $\spec \ti R\to\spec R$ with $f_-$; cf. comments after 
Definition \ref{def-master1}.)
Then $\ti e\cdot \ti f_-=\ti f_m$.

Let $(E_m,\sigma_m)$ 
and  $(E_-,\sigma_-)$ be families on $S\times\spec\ti R$ that are
the pull-backs of the universal family of $X$ via 
$\ti f_m$ and $\ti f_-$, respectively. 
Let $\iota: S\times \ti \xi_0\to S\times \spec \ti R$ be the closed embedding.
Because $\ti e\cdot \ti f_-=\ti f_m$, the families $E_m$ and $E_-$ restricted to $S\times \ti \xi$ are isomorphic;
let $\alpha: E_-|_{S\times \ti\xi}\to E_m|_{S\times\ti\xi}$ be such an isomorphism.
By the property of the line bundle $L(\cdot)$, $\alpha$ induces an isomomorphism 
$L(\alpha):  L(E_-)|_{\ti \xi}\to L(E_m)|_{\ti \xi}$.

Recall that $\sigma_m: L(E_m)\to \sO_{\spec \ti R}$ and $\sigma_-: L(E_-)\to \sO_{\spec \ti R}$
are isomorphisms. By the construction of the $\csta$-action on $X$, $\ti e\cdot \ti f_-=\ti f_m$ implies
that $\ti e\cdot \sigma_-|_{\ti\xi}=\sigma_m|_{\ti\xi}$.
Therefore,
\beq\label{ee}
\sigma_m|_{\ti \xi}\circ L(\alpha)\circ  (\sigma_-|_{\ti \xi})^{-1}=\ti e: \cO_{\spec\ti K}\lra \cO_{\spec\ti K},
\eeq
which is $\ti\beta\cdot\ti\zeta^r$ with $r<0$.
On the other hand, by scaling $\alpha$ by a power of $\ti\zeta$, say $\ti\zeta^d$, we can assume that
$\ti\zeta^d\cdot \alpha$ extends to an
$\ti\alpha: E_-\to E_m$
so that $\iota\sta \ti\alpha: \iota\sta E_-\to\iota\sta E_m$ is non-trivial.
By \eqref{ff}, 
$\iota\sta E_m\cong\iota\sta E_-$ and they are $\tau_-$-stable. 
Thus $\iota\sta \ti\alpha$ is an isomorphism.
Therefore, $\ti\alpha$ is an isomorphism and 
$$\sigma_m\circ L(\ti\alpha)\circ  \sigma_-^{-1}: \cO_{\spec\ti R}\lra \cO_{\spec\ti R}
$$
is an isomorphism.

Finally, since $L(\alpha)$ is independent of scaling, 
$$\sigma_m|_{\ti \xi}\circ L(\alpha)\circ  (\sigma_-|_{\ti \xi})^{-1}=
\sigma_m|_{\ti \xi}\circ L(t^d\cdot \alpha)\circ  (\sigma_-|_{\ti \xi})^{-1}=
\sigma_m|_{\ti \xi}\circ L(\ti\alpha)\circ  (\sigma_-|_{\ti \xi})^{-1}.
$$
Since the right hand side extends to an invertible element in $\ti R$,
it contradicts \eqref{ee} and $r<0$. This proves $l=m$.

We now finish the proof of (a). Let $\ti h:\spec\ti K\to \csta$ be given by
$\ti h\sta(t)=\beta\cdot t^b$ with $0<b<\imath(\ti R)\cdot a$. We let $\ti\varphi_b: \spec \ti R\to S_k$ be
the extension of $\ti h\cdot \ti \varphi_1: \spec K\to
S_k$. Since $\ti g\cdot\ti \varphi_1$ extends to $\ti \varphi_-$ with $\ti\varphi_-(\ti\xi_0)\in B_l$,
$0<b<\imath(\ti R)\cdot a$ guarantees that  $\ti\varphi_b$ exists and $\ti \varphi_b(\ti\eta_0)\in B_2\cup\cdots\cup B_{m-1}$. 
Since for $2\le i\le m-1$, $B_i\sub \ti V$ and $F_k(B_i)=F_k(q_1)$, $\rho\circ F_k\circ\ti \varphi_b$ is defined
and $\rho\circ F_k\circ\ti\varphi_b(B_i)=z_0\in X^\csta$. Since $\ti h\cdot \ti f_+=\rho\circ F_k\circ(\ti h\cdot \ti\varphi_1)
:\spec\ti K\to X$, $\ti h\cdot \ti f_+$
extends to $(\ti h\cdot \ti f_+)^\sim\defeq \rho\circ F_k\circ\ti\varphi_b$ such that $(\ti h\cdot \ti f_+)^\sim(
\ti\xi_0)={z_0}\in X^{\csta}$. This proves (a).

The proof of (b) is similar and we omit it.
\end{proof}

%

\subsection{Symmetric obstruction theory}
In this subsection, we discuss symmetric obstruction theory for
simple flips. Let $X$ be a \DM stack acted on by $\CC^*$. Let
$\cM=[X/\CC^*]$ be the quotient stack. 

\begin{defi}\label{def-perobth}(\cite{GrPand}) We say the quotient stack
$\cM=[X/\CC^*]$ has a \emph{perfect obstruction theory} (resp. \emph{symmetric obstruction theory}) if $X$ has
a $\CC^*$-equivariant perfect (resp. symmetric) obstruction theory.
\end{defi}

We recall that as part of the definition, there is a
$\CC\sta$-equivariant derived category object $F\bul\in
D^b_{\CC^*}(\mathrm{Coh}X)$, which \'etale locally is
quasi-isomorphic to a two-term complex of locally free sheaves,
and an arrow $F\bul \to \LL_X\gmo$ satisfying the requirement of
perfect obstruction theory or symmetric obstruction theory. (Here
$\LL_X\gmo$ is the truncation of the cotangent complex of $X$.)



\begin{rema}\label{rem-CS}
For moduli of stable \emph{sheaves}, analytic locally we may find a function $f$
such that locally the moduli space and its obstruction theory are given by the vanishing of
the differential $df$ (cf. \cite{Joyce}). We call such functions $f$
local \emph{Chern-Simons functionals} of the moduli space. For
more general moduli problems like stable \emph{objects} in the
derived category, local Chern-Simons
functionals are not known to exist. See \cite[Appendix A]{MPT} for a discussion about the
existence of Chern-Simons functional.
\end{rema}


\def\ti{\tilde}
\def\boldw{{\mathbf w}}
\def\boldu{{\mathbf u}}
\def\boldv{{\mathbf v}}
\def\eps{\epsilon}
\def\veps{\varepsilon}
\def\blup{{\rm{bl}}}
\def\hatU{{\hat U}}
\def\hatW{{\hat W}}
\def\umv{^{\text{mv}}}
\def\bv{{\mathbf v}}


\begin{lemm}\label{p2.8} Let $S$ be a  Calabi-Yau three-fold.
Let $M_\pm=[X_\pm/\CC^*]$ be a simple wall crossing in
$\cM=[X/\CC^*]$ of moduli of stable objects in $D^b(\mathrm{Coh}
S)$ (Definition \ref{assum1}) such that $Ext^{\le -1}(E,E)=0$ for
any $E\in \cM$. Suppose either $X$ $\csta$-equivariantly embeds in
a smooth DM-stack, or $X$ has a global tautological family arising
from $X\to \cM$. Then $X$ has a $\CC^*$-equivariant symmetric
obstruction theory.
\end{lemm}

\begin{proof}
%
Suppose first that we have a tautological object $\cE\in
D^b(\mathrm{Coh} (X\times S))$ induced from $X\to \cM$. Let
$\pi:X\times S\to X$ be the projection. Let $\LL_X^\bullet$ denote
the cotangent complex of $X$ and $\LL_\cM^\bullet$ be defined by
the distinguished triangle
\[ \LL_\cM^\bullet\lra \LL_X^\bullet\lra \cO_X,
\]
where the last arrow is given by the action of $\CC^*$ on $X$.
Taking the duals, we have a distinguished triangle
\[ \cO_X\lra {\LL_X^\bullet}^\vee \lra {\LL_\cM^\bullet}^\vee\lra \cO_X[1]. \]

By the construction in \cite{HuyLeh} using the Atiyah class,
we have a morphism ${\LL_X^\bullet}^\vee \to
R:=R\pi_*\cH om(\cE,\cE)_0[1]$. Since the $\CC^*$ action on a pair
$(E,\sigma)\in X$ fixes $E$, the composition
$$\cO_X\lra {\LL_X^\bullet}^\vee\lra R$$
is trivial and hence we have an induced morphism
${\LL_\cM^\bullet}^\vee\to R$. On the other hand, the functorial
assignment $E\rightsquigarrow L(E)$ of Definition \ref{def-det}
induces a morphism
\[
R=R\pi_*\cH om(\cE,\cE)_0[1] \lra Hom(L(\cE),L(\cE))[1]=\cO_X[1] .
\]
Therefore we have a diagram
\[
\xymatrix{ \cO_X \ar@{=}[d]\ar[r] & {\LL_X^\bullet}^\vee \ar[r] &
{\LL_\cM^\bullet}^\vee \ar[r]\ar[d] & \cO_X[1]\ar@{=}[d]\\
\cO_X &&R\ar[r] &\cO_X[1]. }\] Let $R'$ be defined by the
distinguished triangle
\[
\cO_X \lra  R' \lra  R\lra \cO_X[1]
\]
so that we have a commutative diagram of distinguished triangles
\[
\xymatrix{ \cO_X \ar@{=}[d]\ar[r] & {\LL_X^\bullet}^\vee
\ar[d]\ar[r] &
{\LL_\cM^\bullet}^\vee \ar[r]\ar[d] & \cO_X[1]\ar@{=}[d]\\
\cO_X \ar[r] & R'\ar[r] &R\ar[r] &\cO_X[1]. }\] Note that
$H^{-1}(R)=Hom_\pi(\cE,\cE)_0$ injects into $\cO_X$ because the
stabilizer $\CC^*$ of an $E_1\oplus E_2\in M_1\times M_2$ acts
nontrivially on $L(E_1\oplus E_2)$. By a simple diagram chase with
the long exact sequence, we find that the second vertical morphism
${\LL_X^\bullet}^\vee\to R'$ is an obstruction theory of $X$ but
it is not perfect since $H^2(R')\cong H^2(R)=Ext^3_\pi(\cE,\cE)_0$
may be non-trivial.

To remove $H^2(R')$, we take the dual $\gamma^\vee:\cO_X[-1]\to
R^\vee$ of the bottom right arrow $\gamma :R\to \cO_X[1]$ of the
above diagram. By the Serre duality $R^\vee\cong R[1]$, we obtain
\[
\cO_X[-2]\mapright{\gamma^\vee[-1]} R^\vee[-1]\cong
R\mapright{\nu} \cO_X[1].
\]
The composition must vanish by a direct check with Serre pairing. From the
distinguished triangle $R'\to R\to \cO_X[1]$, we therefore obtain
a morphism $\cO_X[-2]\to R'$. Let $R''$ be defined by the
distinguished triangle
\[
\cO_X[-2]\lra R'\lra R'' \lra \cO_X[-1].
\]
Since $H^{-1}(R)=Hom_\pi(\cE,\cE)_0$ injects into $\cO_X$, $\cO_X$
surjects onto $H^2(R')\cong H^2(R)$. From the long exact sequence,
we find that the composition
\[
{\LL_X^\bullet}^\vee\lra R' \lra R''
\]
is a perfect obstruction theory for $X$. By choosing \'etale
locally a free resolution of $R''$ and truncating as in the proof
of \cite[Lemma 2.10]{PanTho}, we can find a two-term complex of
locally free sheaves representing $R''$. It is straightforward to
see that the Serre duality $R\cong R^\vee[-1]$ induces an
isomorphism $R''\cong {R''}^\vee[-1]$ and thus the perfect
obstruction theory ${\LL_X^\bullet}^\vee \to R''$ is symmetric.
This proves the Lemma in case a global tautological family exists.

We next consider the case where $X$ embeds $\csta$-equivariantly into a smooth
DM-stack, say $Y$. Let $I$ be the ideal sheaf of $X\sub Y$.
Then $\LL\gmo_X=[I/I^2\to \Omega_Y|_X]$. Let $\coprod\lalp Y\lalp\to Y$ be an
affine \'etale atlas, let $U\lalp=Y\lalp\times_Y X \to X$, so that $U\lalp$ has tautological family
$\cE\lalp$. The homomorphism $\Omega_X\to\cO_X$
induced by the group action on $X$ lifts to
$\Omega_{U\lalp}\to\cO_{U\lalp}$. Mimicking the prior argument, we
see that $U\lalp$ has a symmetric obstruction theory. By
\cite{Beh-symm}, the symmetric obstruction theory is given by an
almost closed 1-form $\omega\lalp\in\Gamma(\Omega_{Y\lalp})$. In
particular, $U\lalp=(\omega\lalp=0)$ and the obstruction is given
by
$$[T_{Y\lalp}|_{U\lalp}\mapright{d\omega\lalp}\Omega_{Y\lalp}|_{U\lalp}]\lra \LL\gmo_{U\lalp}.
$$

We remark that if $\omega\lalp'$ is another choice of almost closed $1$-form
giving the symmetric obstruction theory, then necessarily $\omega\lalp-\omega\lalp'\in I\lalp^2$,
where $I\lalp$ is the ideal sheaf of $U\lalp\sub Y\lalp$. Thus $d\omega\lalp|_{U\lalp}=d\omega\lalp'|_{U\lalp}$.
Consequently, all $d\omega\lalp$ patch to form a single complex
$[T_{Y}|_X\to \Omega_Y|_X]$, and the arrows above patch to form a single arrow
$$[T_{Y}|_X\mapright{}\Omega_{Y}|_X]\lra \LL\gmo_{X}.
$$
This provides a symmetric obstruction theory of $X$.
\end{proof}


\section{$\csta$-Intrinsic blow-up}
\def\cblow{\mathrm{bl}^{\CC\sta\!}}

Suppose $X$ has a symmetric obstruction theory, the de facto
virtual dimension of $[X/\CC^*]$ is $-1$. This counters our
intuition that both $M_\pm\sub [X/\CC^*]$ should have virtual
dimension $0$.

As an example, suppose $X\sub Y$ is a $\CC^*$-equivariant
embedding in a smooth scheme and $X=(\omega=0)$, where
$\omega\in\Gamma(\Omega_Y)$ is $\CC^*$-invariant. Then the
obstruction theory of $X$ is given by the complex
$$\cO_X(T_Y)\mapright{d\omega}\Omega_Y|_X.
$$
Suppose $\CC^*$ acts on $X$ without fixed points. Then it induces a homomorphism
\beq\label{sigma-3}\eta:
\cO_X\mapright{} \cO_X(T_Y), \eeq whose cokernel is the pull-back
of the tangent sheaf of $X/\CC^*$.
In principle, we expect that the obstruction complex of $X/\CC^*$ should be the
descent to $X/\CC^*$ of
$$\coker\{ \eta\} \mapright{d\omega} \ker\{ \eta\dual\}.
$$
In this way, the obstruction theory of $X/\CC^*$ remains
symmetric.

This argument breaks down near $X^{\CC\sta}$, where $\sigma$ in \eqref{sigma-3} is not 
a subline bundle.
To salvage this argument, we blow up $Y$
along $Y^{\CC^*}$ and work with a ``modified total
transform'' of $X$ --- the modification is to make the
resulting scheme independent of the embedding $X\sub Y$. We
will call such process the \emph{$\CC\sta$-intrinsic blow-up} of $X$.



\subsection{$\csta$-Intrinsic blow-up}

We begin with the easiest
case --- the formal case.
\bigskip

\noindent {\bf The formal case}: {\sl Let $U$ be a formal
$\CC^*$-affine scheme such that its fixed locus $U^{\CC\sta}$ is
an affine scheme, and the set of closed points satisfy
$\text{Set}(U)={\text Set}(U^{\CC\sta})$. }

Since $U^{\CC\sta}$ is affine, we can embed it into a smooth
affine scheme: $U^{\CC\sta}\sub V_0$; since
$\text{Set}(U)=\text{Set}(U^{\CC\sta})$, we can further find a
smooth formal $\CC\sta$-scheme $V$ such that the $\CC^*$-fixed locus $V^{\CC\sta}=V_0$,
$\text{Set}(V)=\text{Set}(V_0)$ and the embedding $U^{\CC\sta}\sub
V_0$ extends to a $\CC\sta$-equivariant embedding $U\sub V$.

Let $\pi:\bar V\to V$ be the blow-up of $V$ along
 $V^{\CC^*}$; let $ \ti U=U\times_V
\bar V\sub \bar V$ be the total transform.
Let $I\sub \cO_V$ be the ideal sheaf defining
$U\sub V$, then $\ti
U$ is defined by the ideal sheaf
$$\ti I\defeq \pi\upmo(I)\cdot \cO_{\bar V}.
$$

We let $E\sub \bar V$ be the exceptional divisor of $\pi$; let
$\xi\in\Gamma(\cO_{\bar V}(E))$ be the defining equation of $E$.
Since $U\sub V$ is $\CC^*$-invariant, $I$ is $\CC^*$-invariant;
thus it admits a weight decomposition $I=I^{\CC^*}\oplus I\umv$,
where $I^{\CC^*}$ is the $\CC^*$-invariant part and $I\umv$
consists of nontrivial weight parts of $I$. Clearly,
$\pi\upmo(I\umv)\sub \xi\cdot \cO_{\bar V}(-E)\sub \cO_{\bar V}$; thus $\xi\upmo
\pi\upmo(I\umv)\sub \cO_{\bar V}(-E)$.

We define $\bar I\sub\cO_{\ti V}$ to be the ideal such that
\beq\label{I-bar} \bar I=\text{the ideal generated by } \pi\upmo(I^{\CC^*})
\ \text{and}\
\xi\upmo \pi\upmo(I\umv)\cdot \cO_{\bar V}(E) . \eeq We
define
$$\cblow_V U \sub {\bar V}
$$
to be the subscheme (of $\bar V$) defined by the
ideal $\bar I$.

\begin{lemm} The scheme $\cblow_V U$ is independent of the choice of
$\CC^*$-equivariant embedding $U\hookrightarrow V$.
\end{lemm}

\begin{proof}
This is a local problem, thus we only need to prove the case where
$U$ is local. (i.e. $U$ contains only one closed point.) We let
$\Gamma(\sO_{V})=A\oplus M$ be the decomposition into the $\csta$
invariant and moving parts; namely  $A=\Gamma(\sO_{V})^{\CC\sta}$
and $M=\Gamma(\sO_{V})\umv$. We let $\Gamma(\sO_{U})=B\oplus N$
similarly with $B=\Gamma(\sO_U)^{\CC^*}$ and
$N=\Gamma(\sO_U)\umv$. Since $V$ is smooth and $A$ is local,
without loss of generality we can assume that
$$\Gamma(\sO_{{V}})=A [\![y_{\leq n}]\!]\defeq A [\![y_1,\cdots, y_n]\!],\quad
y_1,\cdots, y_n\in \Gamma(\sO_{{V}})\umv.
$$
Since $U\to{{V}}$ is a $\CC^*$-equivariant embedding, we can factor
the surjective $\Gamma(\sO_{{V}})\to \Gamma(\sO_{{U}})$
as composition
$$\Gamma(\sO_{{V}})=A[\![y_{\leq n}]\!] \lra B[\![y_{\leq n}]\!] \mapright{\psi_n} \Gamma(\sO_U)=B\oplus N.
$$
Here $\psi_n$ is induced by the restriction $y_i\mapsto \psi_n(y_i)\in N$.

We next pick a minimal set of generators. Let
\beq\label{m} m=\dim
N/(N^2, \fp N).
\eeq
We pick
$\alpha_1,\cdots,\alpha_m\in\Gamma(\sO_{{V}})\umv$ such
that the homomorphism
$$\psi_m: B[\![x_{\leq m}]\!]\defeq B[\![x_1,\cdots,x_m]\!]\lra B\oplus N,\quad x_i\mapsto \alpha_i,
$$
is a surjective ring homomorphism. We let $W=\spec B[\![x_{\le m}]\!]$.
Then $\psi_m$ defines an embedding $ U\sub  W$,
hence an $\CC\sta$-intrinsic
blow-up $\cblow_{ W}U$.

We pick a homomorphism $\varphi$ (as shown in the diagram)
\beq\label{diag-5}
\begin{CD}
B[\![y_{\leq n}]\!] @>{\psi_n}>> B+N\\
@AA{\varphi}A @|\\
B[\![x_{\leq m}]\!] @>{\psi_m}>> B+N
\end{CD}
\eeq By the the choice of $m$ \eqref{m}, $\varphi$ is injective.

We claim that $\varphi$ induces an isomorphism
\beq\label{iso-mn} \Phi(\varphi): \cblow_{ V}U\lra
\cblow_{ W}U, \eeq and that it is
independent of the choice of $\varphi$.

We first simplify the notation further. We let $a_{ij}\in B$ be so
that
$$\varphi(x_i)=\sum_{j=1}^n a_{ij} y_j \mod (y_1,\cdots,y_n)^2
$$
and that the $m\times m$ matrix $(a_{ij})_{1\le i,j\leq m}$ is
invertible. This is possible by the minimality of $m$ and possibly after
re-indexing the $y_j$, since $A$ is local. Thus by replacing $y_j$ by $\varphi(x_j)$ for
$j\leq m$, we can assume $\varphi(x_j)=y_j$. Under this
arrangement, we see that if we let
$$\cI_n\defeq \psi_n\upmo(0)\sub B[\![y_{\le n}]\!]\and
\cI_m\defeq \psi_m\upmo(0)\sub B[\![x_{\le m}]\!],
$$
then there are $h_i\in B[\![x_{\leq m}]\!]\umv$ such that
$$\cI_n=\bl \{\varphi(g)\}_{g\in \cI_m}\cup
\{ y_{m+i}-\varphi(h_i)\}_{1\leq i\leq n-m}\br\sub  B[\![
y_{\le n}]\!].
$$

It is easy to describe the $\CC\sta$-intrinsic blow-ups. First,
$\cblow_{\hat V}\hat U$ is covered by affines
$$R^n_i\defeq \spec B[v^i][\![\zeta^i]\!]/(v_i^i-1), \quad v^i=(v^i_1,\cdots,v^i_n)
$$
with transitions $R^n_i\to R^n_j$ given by the relation
$y_k=v^i_k\zeta^i$. The other
$\cblow_{\hat W}\hat U$ is covered by
affines $R^m_i=\spec B[u^i][\![\xi^i]\!]/(u_i^i-1)$ similarly
defined with $v^i$ replaced by $u^i=(u^i_1,\cdots,u^i_m)$. To
describe the defining equations of $R^m_i\cap
\cblow_{\hat W}\hatU$, we introduce the operation of
substitution $\blup_i$: for $g=g(y_1,\cdots,y_n)\in
B[\![y_{\le n}]\!]$ with $g(0)=0$, we define
$$\blup_i(g)
\defeq (\zeta^i)^{-1}
g(v^i_1\zeta^i,\cdots,v^i_n\zeta^i)|_{v^i_i=1}\in B[v^i][\![\zeta^i]\!]/(v_i^i-1).
$$
The subscheme $R^m_i\cap \cblow_{\hat W}\hatU$ is then
defined by the ideal
$$\blup_i(\cI_m)=\bl\{ \blup_i(g): g\in\cI_m\}\br.
$$
This shows that
$$R^m_i\cap \cblow_{\hat W}\hatU=\spec B[u^i][\![\xi^i]\!]/(u_i^i-1,\blup_i(\cI_m));
$$
similarly for $R^n_i\cap  \cblow_{\hat V}\hatU$.

The homomorphism $\varphi(x_i)=y_i$ defines a homomorphism
\beq\label{Phi-1} \Phi_i(\varphi):
B[u^i][\![\xi^i]\!]/(u_i^i-1,\blup_i(\cI_m))\lra
B[v^i][\![\zeta^i]\!]/(v_i^i-1, \blup_i(\cI_m)) \eeq via
$\xi^i\mapsto \zeta^i$ and $u^i_j\mapsto v^i_j$. Since
$y_{m+j}-\varphi(h_j)\in\cI_n$,
$v^i_{m+j}-\blup_i(\varphi(h_j))\in\blup_i(\cI_n)$. This proves
that $\Phi_i(\varphi)$ is surjective. The same argument shows that
$\Phi_i(\varphi)$ is injective. Thus it is an isomorphism.

We next show that \beq\label{inc-2} \cblow_{ V}U \sub
R^n_1\cup\cdots\cup R^n_m. \eeq Suppose that there is a closed
$b\in \cblow_{ V}U\cap\bl R^n_{m+j}-R^n_1\cup\cdots\cup R^n_m\br$.
Let $[b_1,\cdots,b_n]$ be the homogeneous coordinates of $b$. Then
$b\not\in R^n_1\cup\cdots\cup R^n_m$ implies that $b_i=0$ for
$i\leq m$. By the relation
$v^{m+i}_{m+j}-\blup_{m+i}(\varphi(h_j))$ for $i\geq 1$, we get
$b_{m+i}=0$ for all $i\geq 1$. This is a contradiction. This
proves the inclusion \eqref{inc-2}.

It is obvious that on restricting to $R^n_i\cap R^n_j$,
$\Phi_i(\varphi)=\Phi_j(\varphi)$. Thus the collection
$\{\Phi_i(\varphi)\}_{i\leq m}$ defines an isomorphism
\beq\label{Phi-3} \ti\Phi(\varphi): \cblow_{ V}U \lra \cblow_{
W}U. \eeq

Finally, we show that the isomorphism $\ti\Phi(\varphi)$ is
independent of the choice of $\varphi$. Notice that this will show
that the construction of $\cblow_{ V} U$ is
independent of the choice of the embedding $\psi_n$. Suppose
$\varphi'$ (in place of $\varphi$ in \eqref{diag-5}) is another
homomorphism making the square commutative. We define
$\Phi(\varphi')$, similar to that of $\Phi_i(\varphi)$, by sending
$x_j=u^i_j\xi^i$ to $\varphi'(x_j)|_{y_k\mapsto v^i_k\zeta^i}$.
Here the subscript $y_k\mapsto v^i_k\zeta^i$ means to substitute
the $y_k$ variable in the power series $\varphi'(x_j)$ by
$v^i_k\zeta^i$. It is direct to check that this defines a
homomorphism $\Phi_i(\varphi')$, like \eqref{Phi-1}.

We claim that for $i\leq m$, $\Phi_i(\varphi)=\Phi_i(\varphi')$.
Indeed, by the commutativity of \eqref{diag-5}, there are $g_i\in
\cI_n$ such that $\varphi(x_i)-\varphi'(x_i)=g_i$. Because
$$\varphi'(x_i)|_{y_k\mapsto v^i_k\zeta^i}
=(\varphi(x_i)+g_i)|_{y_k\mapsto v^i_k\zeta^i} \equiv \zeta^i\mod
\zeta^i\blup_i(\cI_n),
$$
we obtain $\Phi_i(\varphi')(\xi^i)=\Phi_i(\varphi)(\xi^i)$. For the
same reason, we have
$\Phi_i(\varphi')(u^i_j)=\Phi_i(\varphi)(u_j^i)$. This proves that
$\Phi_i(\varphi)=\Phi_i(\varphi')$, which implies that the
isomorphism \eqref{Phi-3} is independent of the choice of $\varphi$.

This proves that the $\cblow_{V}U$ is independent of
the embedding $U\sub V$.
\end{proof}

Because $\cblow_VU$ is independent of the embedding $U\sub V$, in the
following we will drop $V$ from the notation and denote by $\cblow U$ the
$\CC\sta$-intrinsic blow up of $U$.

\bigskip

\noindent
{\bf The affine case}: {\sl Let $Y$ be an affine scheme, let $Y_0\sub Y$ be
a closed subscheme and let $\hat Y$ be the formal completion of $Y$ along $Y_0$.
We assume $\hat Y$ is a $\CC\sta$-scheme so that $(\hat Y)^{\CC\sta}=Y_0$ as schemes.
}

The pair $Y_0\sub\hat Y$ is the pair studied in the formal case.
Thus we can form a $\CC\sta$-intrinsic blow-up $\cblow \hat Y$. We
now show that we can glue $\cblow \hat Y$ with $Y-Y_0$ to form a
$\CC\sta$-intrinsic blow-up of $Y$.

We let $A=\Gamma(\cO_Y)$, $I\sub A$ the ideal defining $Y_0\sub
Y$; thus $\hat A=\lim A/I^n=\Gamma(\cO_{\hat Y})$. We cover
$\cblow\hat Y$ by $\CC\sta$-invariant affine $Z_i$; we denote
$B_i=\Gamma(\cO_{Z_i})$, and $\xi_i\in B_i$ the element defining
the exceptional divisor of $\pi_i: Z_i\to \hat Y$. For each $i$,
we glue $Y$ and $Z_i$ as follows. Let $\hat I=\lim I/I^n$ be the
ideal defining $Y_0\sub \hat Y$. We form the localization $\hat
A_{\hat I}$ and $(B_i)_{(\xi_i)}$. Without loss of generality, we
can assume that for each $i$, $\hat A_{\hat I}=(B_i)_{(\xi_i)}$.

For each $i$, we form the direct sum module $B_i\oplus A$ and
define the ring structure: $(b,a)\cdot(b',a')=(bb',aa')$. We then
define $\ti B_i$ to be the kernel
$$\ti B_i=\ker\{B_i\oplus A\lra \hat A_{\hat I}\},
$$
where $A\to \hat A_{\hat I}$ is the composite of the tautological $A\to \hat A$ with the localization homomorphism;
$B_i\to \hat A_{\hat I}$ is the negative of the local homomophism $B_i\to (B_i)_{(\xi_i)}=\hat A_{\hat I}$.

It is routine to check that $\spec \ti B_i$ patch to form a scheme together with a morphism
$$\pi: \cblow Y\lra Y;
$$
it satisfies $\cblow Y\times_{Y}\hat Y=\cblow \hat Y$. The scheme $\cblow Y$ is the $\CC\sta$-intrinsic blow-up
of $Y$ along $Y_0$.
By the independence on the embedding proved in the formal case,
it is canonically defined based on the $\CC\sta$-structure of $\hat Y$.

\bigskip
\noindent
{\bf The scheme case}: {\sl Let $Y_0\sub Y$ be as in the previous case except that it is no longer  assumed
to be affine. Let $\hat Y$ be the formal completion of $Y$ along $Y_0$.
We assume $\hat Y$ is a $\CC\sta$-scheme so that $(\hat Y)^{\CC\sta}=Y_0$.
}

We cover $Y$ by affine open $Y_i\to Y$; we let $Y_{i,0}=Y_i\cap
Y_0$, and let $\hat Y_{i}$ be the formal completion of $Y_i$ along
$Y_{i,0}$. Since $\hat Y_i=\hat Y\times_YY_i$, it is a
$\CC\sta$-scheme. We let $\cblow Y_i$ be its $\CC\sta$-intrinsic
blow up. For each pair $(i,j)$, using that $\hat
Y_i\times_YY_j=\hat Y_j\times_Y Y_i$ as $\CC\sta$-schemes, we have
a canonical isomorphism $\cblow Y_i\times_Y Y_j=\cblow Y_j\times_Y
Y_i$. This shows that we can patch $\cblow Y_i$ to form a scheme,
which we call the $\CC\sta$-intrinsic blow-up of $Y$, and denote
by $\cblow Y$.

\bigskip

\noindent {\bf The DM-stack case}: {\sl We let $X$ be a DM-stack
with a $\CC\sta$-action. We assume that the multiplication morphism
$\si: \CC\sta\times X\to X$ is representable. }

We let $Y$ be a scheme and $Y\to X$ be an \'etale morphism.
We let $Y_0=Y\times_X X^{\CC^*}$, and let $\hat Y$ be the formal
completion of $Y$ along $Y_0$. Using that the multiplcation
morphism $\si$ is representable, one checks that the
$\csta$-action on $Y$ induces a $\csta$-action on $\hat Y$ such
that $\hat Y^{\CC\sta}=Y_0$. Therefore, we can form the
$\CC\sta$-intrinsic blow-up $\cblow Y$.

This construction is canonical. Let $Y'\to X$ be another \'etale
chart of $X$, and let $\cblow Y'$ be the $\CC\sta$-intrinsic
blow-up of $Y'$. Then $Y\times_X Y'$ is \'etale over $X$, and has
its $\CC\sta$-intrinsic blow-up. One checks that $\cblow
(Y\times_X Y')$ is canonically isomorphic to both $\cblow
Y\times_XY'$ and $\cblow Y'\times_XY$.

The $\csta$-intrinsic blow-up $\cblow X$ is the stack that comes
with covers $\cblow Y_i$ and products \beq\label{prod} \cblow
Y_i\times_{\cblow X}\cblow Y_j=\cblow(Y_i\times_X Y_j). \eeq


\begin{prop}
Let $X$ be a $\CC\sta$-equivariant \DM stack as stated, and let
$Y_i\to X$ be an \'etale atlas of $X$. Then the collection $\cblow
Y_i$ together with the morphisms $\cblow(Y_i\times_X Y_j)\to
\cblow Y_i$ and $\cblow(Y_i\times_X Y_j)\to \cblow Y_j$ form a
DM-stack $\cblow X$ with a $\csta$-equivariant projection $\pi:
\cblow X\to X$.
\end{prop}

\begin{proof}
The proof follows from that the $\CC\sta$-intrinsic blow up construction is canonical.
\end{proof}

We call $\pi:\cblow X\to X$ the $\CC\sta$-intrinsic blow-up of $X$.

\subsection{Obstruction theory of $\bar X$--local theory}\label{section-redobth}

We now investigate the obstruction theory of $\bar X$ near the exceptional divisor of $\bar X\to X$.
For notational simplicity, in the following we let $T=\csta$ and call the $\csta$-actions on
$X$ and $\bar X$ the $T$-actions.

We begin with the local situation.
\vsp

\noindent
{\bf The set-up}: {\sl
We let
$X\sub Y=\spec A[\![x]\!]$, where $\spec A$ is smooth and $x=(x_1,\cdots,x_m)$,
be a closed subscheme in smooth affine scheme. We let ${T}$ acts on $Y$ so that it acts on $A$ trivially and on $x_j$ via
$x_j^\sigma=\sigma^{l_j} x_j$, $l_j\ne 0$. We suppose $X\sub Y$ is ${T}$-invariant and
admits a ${T}$-equivariant symmetric obstruction theory.
}

Because $X$ has a symmetric obstruction theory,
by \cite{Beh-symm}, the embedding $X\sub Y$ and its obstruction theory is
defined by the vanishing of an almost closed
1-form
\begin{equation}\label{locdefeq2}\omega=\alpha+\sum_{j=1}^m f_j\, dx_j\in \Omega_{Y},
\quad \alpha\in \Omega_A\otimes_A A[\![x]\!]\and f_j\in A[\![x]\!].
\end{equation}
(Almost closed means $d\omega|_X=0$.) Because the obstruction
theory is ${T}$-equivariant, $\omega$ can be chosen to be
$\CC^*$-invariant. If we let $I_X$ be the ideal sheaf of $X\sub
Y$, then the obstruction theory is given by \beq\label{X-symm}
V\bul\defeq [\cV\mapright{d\omega}\cV\dual]\lra
\LL\gmo_X=[I_X/I_X^2\to \Omega_Y|_X], \quad \cV=T_Y|_X. \eeq

We remark that if $\omega'$ is another ${T}$-invariant almost closed 1-form defining $X\sub Y$ and its
symmetric obstruction theory, then necessarily $\omega-\omega'\in I_X^2\cdot\Omega_Y$.

\vsp We now turn our attention to the ${T}$-intrinsic blow-up
$\bar X$ of $X$. Since $Y$ is smooth, the ${T}$-intrinsic blow-up
$\bar Y$ of $Y$ is the usual blow-up of $Y$ along $Y^{{T}}$. Let
${\bar\pi}: \bar Y\to Y$ be the projection, let $\bar E\sub \bar
Y$ be the exceptional divisor, and let $\xi\in\Gamma(\cO_{\bar
Y}(\bar E))$ be the defining equation of $\bar E\sub\bar Y$.

We consider the pull-back
$${\bar\pi}\sta\omega= {\bar\pi}\sta\alpha+\sum
{\bar\pi}\sta f_j\,dx_j\in\bar\pi\sta\Omega_{\bar Y}.
$$
Since $\omega$ is $T$-invariant and $dx_i$ has non-zero weights,
$\xi\upmo{\bar\pi}\sta f_j\, dx_j$ are
regular sections in ${\bar\pi}\sta\Omega_{Y}\umv(-\bar E)$.

By definition, $\bar X\sub\bar Y$ is defined by the
vanishing of
${\bar\pi}\sta \alpha$ and $\xi\upmo{\bar\pi}\sta f_j$'s.
To put these into a compact form, we introduce
$\eps_j=\bar\pi\sta dx_j$, which span the sheaf
$\bar\pi\sta\Omega_{Y}\umv$, where $\Omega_Y\umv$ is the subsheaf of $\Omega_{Y}$ spanned by $dx_j$'s.
We let $\Omega_{Y}^{\text{fix}}=\Omega_A\otimes_A A[\![x]\!]$ be the fixed part,
and let $\bar\cV_{-1}$ be the sheaf whose dual is
\begin{equation}\label{eq3.09}\bar \cV_{-1}\dual \defeq
{\bar\pi}\sta\Omega_{Y}^{\text{fix}}\oplus
{\bar\pi}\sta\Omega_{Y}\umv(-\bar E) \equiv
\ker\{{\bar\pi}\sta\Omega_{Y}\lra {\bar\pi}\sta(\Omega_{
Y}\umv)|_{{Y}^{T}}\}.
\end{equation}
Then
\beq\label{eq-2010-1} \bar\omega={\bar\pi}\sta \alpha+\sum
\xi\upmo{\bar\pi}\sta f_j \cdot \eps_j\in\Gamma(\bar Y, \bar \cV_{-1}\dual).
\eeq

This way, $\bar X$ has the form
\beq\label{hatbarx} \bar
X=(\bar\omega=0)\sub\bar Y;
\eeq
its obstruction theory is
\beq\label{eq3.2}\bar V\bul\defeq
[\bar\cV_{-1}\mapright{d\bar\omega\dual} \bar \cV_0]
\lra \LL_{\bar X}\gmo=[I_{\bar X}/I^2_{\bar X}\to\Omega_{\bar Y}|_{\bar X}],\quad \bar\cV_0=\Omega_{\bar{Y}}|_{\bar X}.
\eeq

\vsp
For future reference, we comment that the complex $\bar V\bul$ can
be constructed directly from the complex $V\bul$.  Let $\cV=T_Y|_X$.
Let $\pi: \bar X\to X$ be the projection and $E\sub \bar X$ be the exceptional divisor.
Let $\ell$ be the tautological line bundle of the exceptional divisor $E$; namely, tautologically
$\ell\sub \pi\sta N_{Y^{T}\!\!/Y}$. Then
\beq\label{V0}
\bar\cV_0\dual|_{\bar X}=\ker\{\pi\sta\cV\to \pi\sta N_{Y^{T}\!\!/Y}/\ell\}.
\eeq

\vsp
It is easy to describe the fixed locus $\bar X^{T}$ and its obstruction theory.
Let $\cN$ be the normal bundle $N_{Y^{T}\!\!/Y}$.
Since $Y$ is smooth, the exceptional divisor $\bar E=\PP\cN$. To each $k\in\ZZ$, we
let $\cN_{(k)}$ be the weight $k$ piece of the ${T}$-decomposition of $\cN$.
Then
$$
\bar Y^T=\coprod_{k} \bar Y^{T,k},\quad \bar Y^{{T},k}=\PP\cN_{(k)}.
$$
Accordingly,
\beq\label{decomp-X} \bar
X^T=\coprod_{k\in\ZZ} \bar X^{T,k},\quad \bar X^{T,k}=\bar X\cap \bar Y^{T,k}.
\eeq

The obstruction theory of $\bar X^{T}$, following \cite{GrPand},
is given by the invariant part of the obstruction theory of $\bar
X$. To proceed, we reindex the $x_j$'s as $x_{k,j}$ indexed by
$(k,j)$ so that the ${T}$-action is $t\cdot x_{k,j}=t^k x_{k,j}$.
Then
$$\omega=\alpha+\sum_{(k,j)}f_{k,j}dx_{k,j}\and
\bar\omega=\bar\pi\sta\alpha+\sum_{(k,j)}\xi\upmo\bar\pi\sta f_{k,j}\cdot \eps_{k,j}.
$$

Restricting to $\bar Y^{{T},k}$, the ${T}$-action on $f_{i,j}$ (resp. $\xi$) has
weight $-i$ (resp. $k$), among all $\xi\upmo \bar\pi\sta f_{i,j}|_{\bar Y^{T,k}}$, the
${T}$-invariant ones are
$$\xi\upmo \bar\pi\sta f_{-k,j}|_{\bar Y^{T,k}},\quad 1\le j\le e_{-k}=\dim \PP\cN_{(-k)}/Y^{T}+1.
$$
Therefore, the obstruction theory of $\bar X^{T,k}\sub \bar
Y^{T,k}$ is induced by $\pi\sta\alpha|_{\bar Y^{{T},k}}$ and
$e_{-k}$ additional equations. In particular,
$$\text{vir}.\dim \bar X^{T,k}=\text{vir}.\dim X^T+\dim \PP\cN_{(k)}/Y^T-e_{-k}=e_k-e_{-k}-1.
$$
\vsp

We get more from this description of the obstruction theory of $\bar X^{{T},k}$.
The defining equations of $\bar X^{{T},k}\sub \bar Y^{{T},k}$ divide into two groups. The
first is
\beq\label{1}
\bar\alpha_0\defeq \bar\pi\sta\alpha|_{\bar Y^{{T},k}}=0.
\eeq
By the construction of $\bar Y^{{T},k}$, $\bar\alpha_0=\bar\pi\sta\alpha|_{x_{i,j}=0}\in A$; it
defines $\PP_X\cN_{(k)}\defeq \PP\cN_{(k)}\times_{Y^{T}}X^{T}$.
The second group is
\beq\label{2}
\bar\beta_{-k,j}\defeq \xi\upmo\bar\pi\sta f_{-k,j}|_{\bar Y^{{T},k}}=0,\quad 1\le j\le e_{-k}.
\eeq
Expanding $f_{-k,j}=\sum_J a_{-k,j}^J x^J$, where $J$ is the usual multi-index convention and
$a_{-k,j}^J\in A$, we see that
$$\bar\beta_{-k,j}=\sum_{1\le i\le e_k}\xi\upmo a_{-k,j}^{-k,i} \bar\pi\sta x_{-k,i},
$$
which involves weight $-k$ linear (in the $x$ variables) terms in
$f_{-k,j}$. The equations $\bar\beta_{-k,j}=0$ defines the relative
obstruction theory of $\bar X^{{T},k}/\PP_X\cN_{(k)}$.

We now put this in the form of arrows in the derived category. Let
$\pi_{(k)}:\bar X^{T,k}\to Y$ be the tautological projection.
Firstly, the fixed part is
$$\bar V\bul|_{\bar X^{{T},k}}^{\text{fix}}=[(\pi_{(k)}\sta \Omega_Y^{\text{fix}})\dual\oplus \pi_{(k)}\sta \cN_{(-k)}(-1)\to
\Omega_{\bar Y^{{T},k}}|_{\bar X^{{T},k}}].
$$
The obstruction theory of $\PP_X\cN_{(k)}$ is
\beq\label{3}
F_{(k)}\bul\defeq [(\bar\pi\sta \Omega_Y^{\text{fix}})\dual|_{\PP_X\cN_{(k)}}\to
\Omega_{\bar Y^{{T},k}}|_{\PP_X\cN_{(k)}}]\lra \LL_{\PP_X\cN_{(k)}};
\eeq
the relative obstruction theory of $\bar X^{{T},k}/\PP_X\cN_{(k)}$ is
\beq\label{4}
E_{(k)}[-1]\defeq \pi_{(k)}\sta \cN_{(-k)}(-1)\lra \LL_{\bar X^{{T},k}/\PP_X\cN_{(k)}}.
\eeq
(Here the arrows in \eqref{3} and \eqref{4} are defined by $\bar\alpha_0$ and $\bar\beta_{-k,j}$, respctively.)

Using the explicit form of the defining equations inducing the
arrows defining the obstruction theories of $\bar X^{{T},k}$,
$\PP_X\cN_{(k)}$ and $\bar X^{{T},k}/\PP_X\cN_{(k)}$, we have the
following commutative diagram of distinguished triangles
\beq\label{diag-dt}
\begin{CD}
F_{(k)}\bul @>>>  \bar V\bul|_{\bar X^{T,k}}^{\text{fix}} @>>>  E_{(k)}[-1] @>+1>>  \\
@VVV  @VVV @VVV\\
 \LL_{\PP_X\cN_{(k)}}|_{\bar X^{T,k}} @>>> \LL_{\bar X^{T,k}} @>>> \LL_{\bar X^{T,k}/\PP_X\cN_{(k)}} @>+1>>.
\end{CD}
\eeq

Finally, we comment that this diagram is independent of the choice
of the defining equation $\omega$. (It seems to depend on the
choice of the embedding $X\sub Y$.) As argued, if $\omega'$ is
another almost closed 1-form like $\omega$, then
$\omega-\omega'\in I_X^2\cdot\Omega_Y$. Since the left and the
right vertical arrows defined by $\omega$ only use $\bar\alpha_0$
and $\bar\beta_{-k,j}$, to show that they are independent of the
choice of $\omega$, we only need to show that 
\beq\label{incl}
\bar\alpha_0-\bar\alpha_0'\in (\bar\alpha_0)^2, \quad \text{and
for all $j$}:\ \ \bar\beta_{-k,j} -\bar\beta'_{-k,j}\in
(\{\bar\beta_{-k,l}\}_{l=1}^{e_{-k}})^2. 
\eeq 
Using the explicit
forms given after \eqref{1} and \eqref{2}, one sees that
\eqref{incl} is true. This shows that \eqref{diag-dt} is independent of
the choice of $\omega$.

\subsection{The virtual normal cone}

The defining equation \eqref{hatbarx} also
provides a canonical embedding of the the normal cone
\beq\label{C} C_{\bar X/\bar Y}\sub \bar\cV_{-1}\dual|_{\bar X}. \eeq

The ${T}$-action induces a homomorphism $\cO_{Y}\to T_{Y}$; its
pull-back $\cO_{\bar Y}\to {\bar\pi}^*T_{Y}$ vanishes simply along
$E$. This gives an injective homomorphism $\cO_{\bar
Y}(E)\hookrightarrow {\bar\pi}^*T_{Y}$ and its surjective dual
\beq\label{eq3.3} {\bar\pi}^*\Omega_{Y}\lra \cO_{\tilde Y}(-E).
\eeq
Since \eqref{eq3.3} factors through the moving part of
${\bar\pi}^*\Omega_{Y}$, using \eqref{eq3.09}, it induces
a surjective homomorphism \beq\label{surj}\bar\si: \bar\cV_{-1}\dual\lra
\cO_{\bar Y}(-2E).
\eeq
We let
$$\ker\{\bar\si\} =\ker\{\bar\si: \bar
\cV_{-1}\dual\to \cO_{\bar Y}(-2E)\}\sub \bar\cV_{-1}\dual.
$$

\begin{lemm}\label{ob-surj}
The support of the cone $C_{\bar X/\bar Y}$ (cf. \eqref{C}) lies in
the subbundle $\ker\{\bar\si\}|_{\bar X} $.
\end{lemm}

The proof relies on the following criterion proved in \cite[Lemma
4.5]{Kiem-Li}. Let $\bar I$ be the ideal sheaf of $\bar X\sub \bar
Y$. \vskip3pt

\noindent {\bf Criterion}: {\sl The support of $C_{\bar X/\bar
Y}\sub \bar\cV_{-1}\dual|_{\bar X}$ lies in $\ker\{\bar\si\}$ if the
following holds: for any closed $\bar o\in \bar X$ and $\bar\varphi: \spec
k[\![u]\!]\to
\bar Y$ with $\bar\varphi(0)=\bar o$, 
$\bar\varphi\sta(\bar \si\circ \bar\omega)\in t\cdot
\bar\varphi\sta \bar I$. } \vskip3pt

Let $I$ be the ideal sheaf of $X\sub Y$; let $F=\sum_{i=1}^m l_i
x_i f_i\in \Gamma(Y,\cO_Y)$. In the following, we denote by
$(\{x_i\}_1^n)$ the ideal $(x_1,\cdots,x_n)$.

\begin{lemm}\label{sublem}
Let $\varphi: \spec \kk[\![t]\!]\to Y$ be any morphism. Then
$$\varphi\sta(F)\in t\cdot \varphi\sta I\cdot (\{\varphi\sta x_i\}_1^n).
$$
\end{lemm}

\begin{proof}
It suffices to prove the case where $A=\kk[\![u]\!]$,
$u=(u_1,\cdots,u_m)$.

Let $\phi_i(t)=\varphi\sta u_i$ and $\psi_i(t)=\varphi\sta x_i$;
let $a$ and $b$ be defined by \beq\label{ab} \varphi\sta
I=(t^a)\and (\{\psi_i(t)\}_1^n)=(t^b). \eeq Then $\varphi\sta
(F)\in t\cdot \varphi\sta I\cdot (\{\varphi\sta x_i\}_i)$ is
equivalent to $\varphi\sta F=\sum l_i \psi_i\cdot \varphi\sta f_i
\in (t^{a+b+1})$.

We introduce
\beq\label{uv} \Phi(t,s)=(\phi_i(t);\psi_j(t) e^{l_j
s})\defeq (\{u_i(t,s)\}_1^n;\{x_{j}(t,s)\}_1^m), \eeq
which is a homomorphism $\kk[\![u,x]\!]\to \kk[\![t,s]\!]$. Then
$$d x_j(t,s)=\psi_j' e^{l_j s} \, dt+l_j\psi_j e^{l_j s}\, ds.
$$
Using that $\omega$ is $\CC^*$-equivariant, we  have
$(f_i)^\si=\si^{-l_i} f_i$. Thus,
$$\Phi\sta\omega=\sum\varphi\sta f_i\cdot  \psi_i' \, dt
+\sum l_i \varphi\sta f_i \cdot \psi_i \, ds.
$$
Therefore, \beq\label{Phi1} d\bl \Phi\sta \omega\br=\bl\sum l_i
\varphi\sta f_i\cdot \psi_i  \br_t\, dt\wedge ds= (\varphi\sta
F)_t\, dt\wedge ds. \eeq (We use the subscript $t$ to denote
$\frac{\partial}{\partial t}$; same with $x_k$.) On the other
hand, using $d\omega=\sum_{ j>k} \bl( f_j)_{x_k}-(f_k)_{x_j})\br\,
dx_k \wedge dx_j$, we obtain \beq\label{Phi2} d\bl \Phi\sta
\omega\br=\Phi\sta d\omega=\sum_{ j>k}\bl \Phi\sta
(f_j)_{x_k}-\Phi\sta (f_k)_{x_j}\br\, d x_k(t,s)\wedge d x_j(t,s).
\eeq Using \eqref{uv}, we calculate
$$dx_i(t,s)\wedge dx_j(t,s)=e^{(l_i+l_j)s}(l_j\psi_i'\cdot \psi_j-l_i\psi_{i}\cdot\psi_{j}')\, dt\wedge ds.
$$
Applying the $\CC^*$-equivariance, \eqref{Phi2} reduces to
\beq\label{Phi3} d\bl \Phi\sta \omega\br=\sum_{ j>k}\bl
\varphi\sta (f_j)_{x_k}-\varphi\sta (f_k)_{x_j}\br\bl
l_j\psi_i'\cdot \psi_j-l_i\psi_{i}\cdot\psi_{j}'\br\, dt\wedge ds.
\eeq

Since $\omega$ is almost closed, $(f_j)_{x_k}- (f_k)_{x_j}\in I$ for
all $j,k$. We then use \eqref{ab} to conclude that
$$\bl \varphi\sta (f_j)_{x_k}-\varphi\sta (f_k)_{x_j}\br\,
(l_j\psi_i'\cdot \psi_j-l_i\psi_{i}\cdot\psi_{j}')\in (t^{a+2b-1}).
$$
Thus equating \eqref{Phi1} and \eqref{Phi3}, we conclude
$(\varphi\sta F)_t\in (t^{a+2b-1})$. Since $\varphi\sta F(0)=0$,
$\varphi\sta F\in (t^{a+2b})$. Since $b\ge 1$, this proves
$\varphi\sta F\in (t^{a+b+1})$. This proves the Lemma.
\end{proof}

\begin{proof}[Proof of Lemma \ref{ob-surj}]
As before, we only need to prove the case where
$A=\kk[\![u]\!]$, $u=(u_1,\cdots,u_m)$.
We verify the criterion. Let
$\bar\varphi: \spec \kk[\![ t]\!]\to \bar Y$, $\bar\varphi(0)=\bar
o\in \bar X$ be any morphism. We will show that
$\bar\varphi\sta(\bar\si\circ\bar \omega)\in t\cdot \bar\varphi\sta
\bar I$. Clearly, we only need to check the case where $\bar o\in
E$.

For the moment, we assume that $\bar\varphi$ does not factor
through $E\sub \bar Y$. Since $\bar Y$ is the blow-up of $Y$ along
$x_1=\cdots=x_m=0$, canonically $\bar Y\sub Y\times\PP^{m-1}$, and
we can choose homogeneous coordinates of $\PP^{m-1}$ so that $\bar
Y$ is defined by $\frac{x_i}{w_i}=\frac{x_j}{w_j}$, for all $i\ne
j$. By reordering the indices of $x_i$, we can assume $\bar o\in
\{w_1\ne 0\}$.

Let $\varphi=\bar\pi\circ\bar\varphi: \spec [\![t]\!]\to Y$. Like
before, we denote $\phi_i(t)=\varphi\sta u_i$ and
$\psi_j(t)=\varphi\sta x_j$. Since $\varphi$ does not factor
through $E$, $\psi_1(t)\ne 0$; $\bar o\in \{w_1\ne 0\}$ implies
that as ideals, \beq\label{ideal}
(t^b)=(\{\psi_i(t)\}_1^n)=(\psi_1(t)). \eeq Then, following the
definition, we have
$$\bar\varphi\sta \bar I=(\{\varphi\sta g_i\}_1^n, \{ \psi_1(t)\upmo\varphi\sta f_j\}_1^m),
$$
and
$$\bar\varphi\sta(\bar\si\circ\bar \omega)=\psi_1(t)\upmo\sum l_i\varphi\sta x_i\cdot \psi_1(t)\upmo \varphi\sta f_i
=\psi_1(t)^{-2}\sum l_i\varphi\sta x_i\cdot \varphi\sta f_i.
$$
By the definition of $a$, we also have $(\psi_1(t)\upmo \cdot
t^a)\sub \bar\varphi\sta \bar I$; by Lemma \ref{sublem} and
\eqref{ideal}, we have \beq\label{sup} \sum l_i\varphi\sta x_i\cdot
\varphi\sta f_i\in t(\psi_1(t)\cdot t^a). \eeq This implies
$\bar\varphi\sta(\bar\si\circ\bar\omega)\in t\cdot \bar\varphi\sta
\bar I$.

It remains to verify the case when $\bar\phi$ factors through $E\sub
\bar Y$. In this case, we can find $\bar\varphi_1$ so that
$\bar\varphi-\bar\varphi_1\in (t^k)$ for a sufficiently large $k$
and $\bar\varphi_1$ does not factors through $E$. Then by what was
proved, $\bar\varphi_1\sta(\bar\si\circ\bar\omega)\in t\cdot
\bar\varphi_1\sta \bar I$. Since $k$ is sufficiently large, we
conclude $\bar\varphi\sta(\bar\si\circ\bar\omega)\in t\cdot
\bar\varphi\sta \bar I$. This proves the Lemma.
\end{proof}


\subsection{The obstruction theory of $\bar X$--global theory}

We begin with the following situation.

We suppose there is an \'etale affine altas $X\lalp\to X$, and two-term complexes of locally
free sheaves $V\lalp\bul\to \LL_{X\lalp}\gmo$ giving the symmetric obstruction theories of
$X\lalp$. (We will ignore the compatibility condition for the moment.)


We let $U\lalp$ (resp. $\hat X$) be the formal completion of $X\lalp$ (resp. $X$)
along $X\lalp\times_X X^{T}$ (resp. $X^{T}$).
Since the ${T}$-action on $X$ is representable, each $U\lalp$ is a ${T}$-scheme
and the tautological $U\lalp\to\hat X$ is ${T}$-equivariant.

We then pick smooth affine ${T}$-schemes $Y\lalp$ and
${T}$-equivariant embeddings $U\lalp\to Y\lalp$. We let $\hat
V\lalp\bul=[\hat \cV\lalp\to\hat
\cV\lalp\dual]\to\LL_{U\lalp}\gmo$ be the pull-back to $U\lalp$ of
$V\lalp\bul\to\LL_{X\lalp}\gmo$. Since
$V\lalp\bul\to\LL_{X\lalp}\gmo$ are symmetric obstruction
theories, by \cite{Beh-symm} and that $T$ is reductive, there are
$T$-invariant almost closed 1-forms $\omega\lalp\in\Gamma(\Omega_{Y\lalp})$ so
that $U\lalp=(\omega\lalp=0)$, and for
$\cW\lalp=T_{Y\lalp}|_{U\lalp}$ there are quasi-isomorphisms that
make the following squares commutative \beq\label{diag3}
\begin{CD} \hat V\lalp\bul @>{\cong_{q.i}}>> [\cW\lalp\mapright{d\omega\lalp\dual} \cW\lalp\dual]\\
@VVV @VVV\\
\LL_{U\lalp} \gmo@= \LL_{U\lalp}\gmo.
\end{CD}
\eeq Here the right vertical arrow is the obstruction theory of
$U\lalp$ induced from $\omega\lalp=0$. By shrinking $X\lalp$ and
altering the dimensions of $Y\lalp$ if necessary, we can assume
the (top line) quasi-isomorphism is an isomorphism.

We let $\bar U\lalp\sub \bar Y\lalp$ be the pair of
${T}$-intrinsic blow-up of $U\lalp\sub Y\lalp$, let $\pi\lalp:
\bar U\lalp\to U\lalp$ be the projection. Following the convention
in the previous subsection, we let $\bar   E\lalp\sub \bar Y\lalp$
be the exceptional divisor, let
$N\lalp=N_{Y\lalp^{T}\!\!/Y\lalp}=\cW\lalp|_{Y\lalp^{T}}\umv$ and let
$\ell\lalp\sub \pi\lalp\sta N\lalp$ be the tautological subline
bundle. We form
$$\bar\cW_{0,\alpha}=\ker\{ \pi\lalp\sta\cW\lalp\to (\pi\lalp\sta N\lalp)/\ell\lalp\}\dual \and
\bar\cW_{-1,\alpha}=\ker\{\pi\lalp\sta\cW\lalp\dual\to \pi\lalp\sta(\cW\lalp\dual|_{Y\lalp^{T}}\umv)\}\dual.
$$
(Note $\bar\cW_{0,\alpha}=T_{\bar Y\lalp}|_{\bar U\lalp}$.)
The induced perfect obstruction theory of $\bar U\lalp$ is given by
\beq\label{loc-ob}
[\bar\cW_{-1,\alpha}\mapright{d\bar\omega\lalp\dual}\bar\cW_{0,\alpha}]\lra \LL_{\bar U\lalp}\gmo.
\eeq

Using the isomorphisms at the top line of \eqref{diag3}, we can glue \eqref{loc-ob} with the
restriction to $\bar X\lalp-\bar E\lalp$ of the top line of \eqref{diag3} to obtain a new complex with an arrow over
$\bar X\lalp$:
\beq\label{U-ob}
\bar V\bul\lalp=[\bar \cV_{-1,\alpha}\to \bar\cV_{0,\alpha}]\lra \LL_{\bar X\lalp}\gmo.
\eeq
This arrow gives the induced perfect obstruction theory of $\bar X\lalp$.

In order that the collection $V\lalp\bul\to\LL_{X\lalp}\gmo$
gives the symmetric obstruction theory of $X$, it must satisfy
certain compatibility condition. Since the construction of the
induced perfect obstruction theory on $\bar X\lalp$ is canonical,
its compatibility largely follows from the compatibility of $V\lalp\bul\to\LL\gmo_{X\lalp}$.
For instance, if $V\lalp\bul\to\LL_{X\lalp}\gmo$ are restrictions of the symmetric obstruction theory
$V\bul=[\cV\to \cV\dual]\to \LL_X\gmo$ of $X$, then $\bar V\lalp\bul$ 
constructed in \eqref{U-ob} patch to form a global complex $\bar
V\bul$ on $\bar X$. However, the arrows $\bar V\bul|_{\bar
X\lalp}\to \LL_{\bar X\lalp}\gmo$ given in \eqref{U-ob} may not
coincide, the reason being that the presentation $\LL_{\bar
X\lalp}\gmo$ uses the auxiliary embeddings $U\lalp\to Y\lalp$.

We comment that this construction yields a {\sl local symmetric obstruction theory}
to be formulated in \cite{Chang-Li}.

\begin{prop}\label{prop3.7}
Let $X$ be a ${T}$-equivariant \DM stack with a ${T}$-equivariant
symmetric obstruction theory. Suppose $X$ has a ${T}$-equivariant
embedding $X\to Y$ into a smooth DM-stack. Then the
${T}$-intrinsic blow-up $\bar X$ has an induced obstruction
theory.
\end{prop}

\begin{proof}
Let $X\to Y$ be the ${T}$-equivariant embedding into a smooth DM-stack.
By Lemma \ref{p2.8},
$X$ has a ${T}$-equivariant symmetric
obstruction theory
\beq\label{VV}
V\bul=[\cV\to\cV\dual]\to\LL\gmo_X=[I_X/I_X^2\to \Omega_Y|_X], \quad \cV=T_Y|_X.
\eeq

We let $\bar X\sub \bar Y$ be the ${T}$-intrinsic blow-up of the pair $X\sub Y$,
let $\pi: \bar X\to X$ be the projection, and let $\bar E\sub \bar Y$ be the exceptional divisor.
We let $\bar \cV_{-1}$ and $\bar\cV_0$ be the two
locally free sheaves on $\bar X$ given in \eqref{eq3.09} and \eqref{V0}.

We cover $X$ by affine \'etale atlas $X\lalp\to X$. Let $\bar X
\lalp$ be the $\csta$-intrinsic blow-up. Over $\bar X\lalp$, the
pull-back of $\bar\cV_{-1}$ and $\bar\cV_0$ are the sheaves
$\bar\cV_{-1,\alpha}$ and $\bar\cV_{0,\alpha}$ mentioned in
\eqref{U-ob}. Let $\bar V\bul\lalp\to\LL_{\bar X\lalp}\gmo$ be as in
\eqref{U-ob}.

As argued, the arrows in \eqref{U-ob} are independent of the choice
of an almost closed 1-form $\omega\lalp$; thus the collection
$\bar V\lalp\to \LL_{\bar X\lalp}\gmo$ patch to form a global
complex and an arrow \beq\label{g-V} \bar
V\bul=[\bar\cV_{-1}\to\bar\cV_0]\lra \LL_{\bar X}\gmo=[ I_{\bar
X}/I_{\bar X}^2\to \Omega_{\bar Y}|_{\bar X}]. \eeq Also, because
the construction is canonical, it is $T$-equivariant. This proves
the Lemma.
\end{proof}

Let $\bar X^{T,k}\sub \bar Y^{T,k}$ ($\sub \bar Y^T$) be the
decomposition of the fixed locus shown in \eqref{decomp-X}.
(Because the inclusion $X\sub Y$ is a global $T$-embedding, the
local constructions patch to form a global decomposition.) Let
$N_{(k)}$ be the weight $k$ component of $N_{Y^T/Y}$, and let
$\PP_X N_{(k)}=\PP N_{(k)}\times_{\bar Y^T}\bar X^T$ be as defined
in the previous subsection. We form complex $F\bul_{(k)}$ and
sheaf $E_{(k)}$, as in \eqref{3} and \eqref{4}. By the remark after
the diagram \eqref{diag-dt}, we have the following Lemma.

\begin{lemm}
The constructions preceding \eqref{diag-dt} patch to give an
obstruction theory of $\PP_X N_{(k)}$ and a relative obstruction
theory of $\bar X^{T,k}/\PP_X N_{(k)}$. They fit into the
commutative diagram of distinguished triangles shown in
\eqref{diag-dt}.
\end{lemm}

By Lemma \ref{ob-surj}, we obtain the following.

\begin{coro}\label{reduction} Let the situation be as in Proposition \ref{prop3.7}.
Then the virtual normal cone $C_{\bar X}\sub \bar\cV_{-1}^\vee$ of
$\bar X$ lies entirely in the kernel of $\bar\cV_{-1}^\vee\to
\cO_{\bar X}(-2\bar E)$.
\end{coro}





\black

\section{The master space}

In this section we define the master space of a simple flip
$$M_\pm =[X_\pm/\CC^*]\subset \cM=[X/\CC^*]$$ and prove that it
is a proper separated \DM stack. We also define and study the
master space of the $\CC\sta$-intrinsic blow-up $\bar X$. This master space will be the
main tool for our wall crossing formula in the subsequent section.

\subsection{Master space} \label{section-master}

Let $M_\pm=[X_\pm/T]\sub\cM$ be the simple flip defined in Definition \ref{def-master1}.
We consider $X\times\Po$ with the $T$-action
$$\sigma\cdot (w,[s_0,s_1])=(\sigma\cdot w,[\sigma\cdot s_0, s_1]).
$$
We pick a $T$-invariant open subset
$$(X\times \Po)^s=X\times \Po-\Sigma_-\times\{0\}-\Sigma_+\times\{\infty\},
$$
where $0= [0,1]$ and $\infty=[1,0]$; we then form the quotient
\begin{equation}\label{cZ}
Z=(X\times \Po)^s/T.
\end{equation}
Obviously, $(X\times \Po)^s$ contains both $X_+\times\{0\}$ and
$X_-\times\{\infty\}$ as closed substacks.

\begin{defi} We call $Z$ the \emph{master space} for
$\cM$. \end{defi}

We will see below that $Z$ is a proper separated
$\CC\sta$-equivariant \DM stack.

\begin{exam}\label{ex2.7}
For the $X$ in Example \ref{ex2}, 
the master space $Z$ is simply the blow-up of $\PP V$ along $\PP
V_+\cup \PP V_-$.
\end{exam}

%

We intend to show that the master space $Z$ is proper. In the
discussion below, we will use $R$ to denote a discrete valuation
ring over $\CC$ with fractional field $K$; denote by $\zeta$ its
uniformizing parameter, and denote by $\xi$ and $\xi_0$ its
generic and closed points. 
Also, for an
$f:\spec K\to X$ and $g:\spec K\to {{T}}$, we denote by $g\cdot f$
the composite
$$\spec R\mapright{(g,f)} {{T}}\times X \lra X,
$$
where the second arrow is the group action morphism.

\begin{lemm}\label{lem2.8}
The master space $Z$ is a proper separated \DM stack.
\end{lemm}

\begin{proof} Let $\cZ=(X\times\PP^1)^s$ so that $Z=[\cZ/T]$.
It is direct to check that the stabilizer of any closed $z\in \cZ$
is finite. Also, all $T$-orbits of $\cZ$ are closed orbits. Thus
$Z$ is a separated \DM stack.

It remains to prove that $Z$ is proper. Let $R$ be a discrete
valuation ring over $\CC$ with field of fractions $K$, and let $f:
\spec K\to \cZ$ be a morphism. We need to show that after a finite
extension $\tilde R\supset R$ with $\tilde K$ its field of
fractions, there is a morphism $g:\spec \tilde K\to {{T}}$ so that
$g\cdot f: \spec\tilde K\to \cZ$ extends to $(g\cdot f)^{ex}:\spec
\tilde R\to \cZ$.

First, note that $\cZ$ decomposes into the disjoint union
$$\cZ=X_0\times{{T}}\sqcup \Sigma_-^\circ\times{{T}}\sqcup
\Sigma_+^\circ\times{{T}}\sqcup X_+\times\{0\}\sqcup
X_-\times\{\infty\}\sqcup X^T\times{{T}}.
$$
Using $\cZ\sub X\times\Po$, we can write
$$f=(f_1,f_2): \spec K\lra X\times\Po.
$$
Let $\xi$ and $\xi_0$ be the generic and closed point of $\spec R$.

We first consider the case $f(\xi)\in X_+\times\{0\}$. Since by
assumption the quotient $[X_+\times\{0\}/T]=[X_+/T]=M_+$ is proper,
the extension $(g\cdot f)^{ex}$ does exist. The case $f(\xi)\in
X_-\times\{\infty\}$ is similar. For the same reason, if $f(\xi)\in
X^T\times{{T}}$, because $X^T$ is proper, the extension also
exists.

We now suppose $f(\xi)\in X_0\times{{T}}$. Because $M_+=[X_+/T]$
is proper, after a finite extension $\tilde R$ of $R$, we can find
a morphism $g_+:\spec \tilde K\to {{T}}$ so that $g_+\cdot f:
\spec \tilde K\to X_0\times{{T}}$ extends to $(g_+\cdot f)^{ex}:
\spec\tilde R\to X_+\times\Po$. By the same reason, after
replacing $\tilde R$ by a finite extension, still denoted by
$\tilde R$, we can find $g_-: \spec \tilde K\to {{T}}$ so that
$g_-\cdot f$ extends to $(g_-\cdot f)^{ex}: \spec\tilde R\to
X_-\times\Po$. Let $g^{ex}_+$ and $g_-^{ex}: \spec\tilde R\to \Po$
be the extensions of $g_+$ and $g_-$. In case
$g_+^{ex}(\tilde\xi_0)\ne \infty$ or $g^{ex}_-(\tilde\xi_0)\ne 0$,
then either $(g_+\cdot f)^{ex}$ or $(g_-\cdot f)^{ex}$ maps to
$\cZ$ and we are done. Suppose not. Let $g:\spec \tilde K\to{{T}}$
be defined via $g_+\cdot g=g_-$, which gives $g\cdot (g_+\cdot
f)=g_-\cdot f$. Because $g_+^{ex}(\tilde\xi_0)= \infty$ or
$g^{ex}_-(\tilde\xi_0)= 0$, we must have $g\sta(t)=\alpha\cdot
\zeta^a$ with $a>0$. Therefore by Lemma \ref{lem-auxprop} (a),
after possibly another finite extension $R\subset \tilde R$, we
can find $g':\spec \tilde K \to T$ such that
$(g')^*(t)=\zeta^{a'}$ with $0<a'<a$ and that $(g'\cdot g_+\cdot
f)^{ex}(\tilde \xi_0)\in X^T\times{{T}}\subset \cZ$.

Finally, we consider the case $f:\spec K\to
\Sigma_+^\circ\times{{T}}$. (The case to
$\Sigma_-^\circ\times{{T}}$ is similar.) Since
$\Sigma_+^\circ\sub X_+$ is closed and $[X_+/T]$ is proper,
$[\Sigma_+^\circ/T]$ is proper. Thus there is a $g: \spec \tilde
K\to{{T}}$, for a finite extension $\tilde R\supset R$, such
that $g\cdot f$ extends to $(g\cdot f)^{ex}: \spec \tilde R\to
\Sigma_+^\circ\times\Po$. In case $(g\cdot f)^{ex}(\tilde\xi_0)\in
\Sigma_+^\circ\times\CC$, we are done. Otherwise, $(g\cdot
f)\sta(t)=\alpha\cdot \zeta^a$ with $a<0$. Then we let $g':\spec
\tilde K\to {{T}}$ be so that $g^{\prime \ast}(t)=\alpha$. By
Lemma \ref{lem-auxprop} (b), possibly after passing through a new
finite extension $\tilde R$, the extension $(g'\cdot f)^{ex}:
\spec \tilde K\to \Sigma_+\times {{T}}$ exists. This settles the
case.

Combining these, we conclude that the quotient $[Z/T]$ is a proper
separated \DM stack.
\end{proof}

Let $H=\csta$ act on $X\times \PP^1$ by
$$(w,[a_0,a_1])^t=(w,[a_0,ta_1]).
$$
Obviously the action of $H$ commutes with the action of $T$ on
$X\times \PP^1$ and hence $Z$ admits an induced action of $H$. The
following lemma is straightforward.
\begin{lemm}
The $H$-fixed point substack $Z^{H}\sub Z$ is the disjoint union of
$$M_+=X_+/T,\quad M_-=X_-/T, \and X^T. 
$$
\end{lemm}

For the obstruction theory, we have

\begin{lemm}
If the quotient stack $\cM=[X/T]$ admits a perfect obstruction
theory, then the master space $Z$ has an induced $H$-equivariant
perfect obstruction theory.
\end{lemm}

\begin{proof}
By definition, $X$ comes with a $T$-equivariant perfect
obstruction theory. The perfect obstruction theory of $X$ lifts to
a $T\times H$-equivariant perfect obstruction theory of
$X\times\Po$, which restricts to the open substack
$(X\times\Po)^s$. Since $T$ acts with only finite stabilizers, the
quotient $Z$, as we saw earlier for $M_\pm$, has an induced
$H$-equivariant perfect obstruction theory. This proves the lemma.
\end{proof}

It follows from the proof that $\vdim Z=\vdim X=\vdim
M_\pm=0$.

\def\hatbarx{\hat{\bar X}}

\subsection{Master space for $\bar X$}
By the same construction as above, we can define the master space
$\bar Z$ for the $\CC\sta$-intrinsic blow-up $\bar X$. We will show that
$M_\pm$ is a subset of the fixed points
of $\bar X$.

Consider $\bar X\times \PP^1$ with $T$-action
$$(w,[s_0,s_1])^\sigma = (w^\sigma,[\sigma\cdot
s_0, s_1])$$ and let
\[
(\bar X\times \PP^1)^s:=\bar X\times \PP^1 - (\sigc_-\cup \bar
E)\times\{0\}-(\sigc_+\cup \bar E)\times \{\infty\}
\]
where $\bar E$ is the inverse image of $X^T$ in $\bar X$.
The master space $\bar{Z}$ is defined as the quotient of the
open stack $(\bar X\times \PP^1)^s$ by the action of $T$. The same
argument as in the proof of Lemma \ref{lem2.8} proves that $\bar Z$
is a proper separated \DM stack. By construction, $\bar Z$ has a
partition
\[
\bar Z = \bar X\sqcup M_+\sqcup M_-,
\]
where $\bar X$ is identified with the quotient of $\bar
X\times(\PP^1-\{0,\infty\})$ by $T$.

The action of $H=\CC^*$ on $\bar{X}\times \PP^1$ by
$(w,[a_0,a_1])^\sigma=(w,[a_0,\sigma\cdot a_1])$ induces an
$H$-action on $\bar Z$. It is straightforward to check that the
$H$-fixed point set in $\bar Z$ is the (disjoint) union \beq
\label{eq3.019} \bar{Z}^H=\bar X^T\sqcup M_+\sqcup M_- . \eeq

When $\cM=[X/T]$ is equipped with a symmetric obstruction theory,
Proposition \ref{prop3.7} gives us a perfect obstruction theory
of $\bar X$ of virtual dimension $0$. By pulling back this obstruction
theory, $(\bar X\times \PP^1)^s$ is equipped with a perfect
obstruction theory of virtual dimension $1$; it then induces a
perfect obstruction theory on the master space $\bar Z$ of virtual
dimension $0$.


\section{A wall crossing formula for symmetric obstruction theories}
In this section, we prove a wall crossing formula for simple flips
with symmetric obstruction theory. Let $M_\pm=[X_\pm/T]\sub\cM$ be a simple flip.
We assume that $X$ embeds
$T$-equivariantly in a smooth DM-stack $Y$.

%


\def\bV{\mathbf{V} }

We consider the projections
\[
\bar X\mapleft{q_-} (\bar X\times
\PP^1)^s\mapright{q_+} \bar Z.
\]
We use the notation of Proposition \ref{prop3.7} and Corollary
\ref{reduction}. The pull-back $q_-^*\bar V\bul$ descends to
the quotient stack $\bar Z$; we denote the descent by $[\sV_{-1}\to\sV_0]$.
Similarly,
$q_-^*\cO_{\bar X}(-2\bar E)$ descends to an invertible sheaf on $\bar
Z$, denoted by $\cO_{\bar Z}(-2\bar E)$.

Let $C_{\bar Z}\sub \sV_{-1}\dual$ be the virtual normal
cone cycle of $\bar Z$, where by abuse of notation we
denote by $\sV_{-1}\dual$ also the vector bundle associated to
$\sV_{-1}\dual$. By Corollary \ref{reduction}, the cycle
$C_{\bar Z}$ lies in the kernel bundle
$$\sV_{-1,\text{red}}\dual=\ker\{ \sV_{-1}\dual\lra \cO_{\bar Z}(-2\bar E)\}.
$$
We let $[\bar Z]\virt=0^!_{\sV_{-1}\dual}[C_{\bar Z}]$ and $[\bar
Z]\virt_{\mathrm{red}}=0^!_{\sV_{-1,\text{red}}\dual}
[C_{\bar Z}]$, as $H$-equivariant cycles. Thus
\[
[\bar Z]\virt=[\bar Z]\virt_{\mathrm{red}}\cap
c_1(\cO_{\bar Z}(-2\bar E)).
\]

Applying the virtual localization theorem \cite{GrPand} to $\bar{Z}$, and using
\eqref{eq3.019}, we obtain \beq [\bar{Z}]\virt = \frac{[\bar X^T
]\virt}{e(N\virt)}-[M_+]\virt - [M_-]\virt.
\eeq
After dividing both
sides by $c_1(\cO_{\bar X}(-2\bar E))$, we obtain
\[
[\bar Z]\virt_{\mathrm{red}}= \frac{[\bar X^T
]\virt}{e(N\virt)c_1(\cO_{\bar
Z}(-2\bar E))}+[M_+]\virt/(-t)+[M_-]\virt/t,
\]
where $t\in A^1_{\CC^*}(pt)$ is the standard generator of
$A^*_{\CC^*}(pt)$. Taking the residue at $t=0$, the left hand side
vanishes because $[\bar Z]\virt_{\mathrm{red}}$, as an equivariant
class, has no negative degree terms in $t$. \black We thus obtain
\beq\label{eq-afwc} \deg [M_+]\virt - \deg [M_-]\virt=\res_{t=0}
\frac{[\bar X^T ]\virt}{e(N\virt)c_1(\cO_{\bar Z}(-2E))}. \eeq

To proceed, we need a description of the cycle $[\bar
X^T]\virt$. We keep the notation introduced after \eqref{V0}. Namely, $\cN=N_{Y^T/Y}$,
$\cN_{(k)}$ is the weight $k$ piece of $\cN$, $\bar Y^{T,k}=\PP\cN_{(k)}$,
and $\bar X^T=\coprod_{k\ne 0} \bar X^{T,k}$, where $\bar X^{T,k}= \bar Y^{T,k}\cap \bar X$.
\footnote{For $p\in X^T$, we let
$$ H^i(p)_j=\text{weight $j$ part of the $T$-decomposition of } H^i(V\bul|_p);
$$
we let $\ell^i_j(p)=\dim H^i(p)_j$ and let
$\delta_j(p)=\ell^0_j(p)-\ell^1_{-j}(p)-1$. We also let
$$\PP^{\delta_j(p)} H^0(p)_j\sub  \PP H^0(p)_j
$$
be a dimension $\delta_j(p)$ linear subspace; it is
the empty set when $\delta_j(p)<0$.
}

As argued in \cite{GrPand}, the obstruction theory of $X^T$ is given by
$[\cV\to \cV^{\vee}]^T\to\LL_{X^T}\gmo$ and has
virtual dimension zero. Thus
\beq\label{XT-vir}
[X^T]\virt=\sum_{i=1}^r a_i [p_i],\quad a_i\in\QQ, \ p_i\in X^T\ \text{closed}.
\eeq
Let $\bar\pi_{(k)}: \PP\cN_{(k)}\to Y^T$ be the projection.

\begin{lemm}\label{lem-vcfix} With the above notation, we have
$$[\bar X^T]\virt=\sum_{1\leq i\leq r,\,k\ne 0}  a_i \cdot e(\bar\pi_{(k)}\sta \cN_{(-k)}\dual(1))
\cap [\PP\cN_{(k)}\times_{Y^T} p_i].
$$
\end{lemm}

\begin{proof}
Using diagram \eqref{diag-dt} and the explicit form $E_{(k)}=\bar\pi_{(k)}\sta \cN_{(-k)}(-1)|_{\bar X^{T,k}}$, 
applying the main result in \cite{KKP}, 
we obtain
$$[\bar X]\virt=\sum_{k\ne 0}[\bar X^{T,k}]\virt=
\sum_{k\ne 0} e(\bar\pi_{(k)}\sta \cN_{(-k)}\dual(1))\cap [\PP_X \cN_{(k)}]\virt.
$$
For $[\PP_X\cN_{(k)}]\virt$, we notice that the obstruction theory of
$\PP_X\cN_{(k)}$ is given by \eqref{3}, and the arrow is given by $\bar\alpha_0$ in \eqref{1}.
Thus using \eqref{XT-vir}, we obtain
$$[\PP_X\cN_{(k)}]\virt=\sum_{1\le i\le r} a_i\cdot [\PP_X\cN_{(k)}\times_X p_i].
$$
This proves the Lemma.
\end{proof}

We now prove our main theorem.

\begin{theo}\label{theo-sym} Let the situation be as in Theorem \ref{thm1.2}.
Let $[X^T]\virt=\sum_{i=1}^ra_i[p_i]$; let
$n_{i,j}$ be the dimension of the weight $j$ part of the tangent space $T_{X,p_i}$ and $n_i=\sum_j n_{i,j}$.
Then
\[
\deg [M_+]\virt - \deg [M_-]\virt=\sum_{i} a_i\cdot
\Bigl((-1)^{n_i-1}\sum_j\frac{n_{i,j}}{j}\Bigr).
\]
\end{theo}

\begin{proof}
By Lemma \ref{lem-vcfix} and \eqref{eq-afwc}, it suffices to show
that for any closed point $p\in X^T$,
\beq\label{eq-ptwcf}
\sum_{j\ne 0}\mathrm{res}_{t=0}\frac{e(\bar\pi_{(j)}\sta \cN_{(-j)}\dual(1))\cap 
[\PP\cN_{(j)}\times_{Y^T} p]}{e(N\virt)\,
c_1(\cO_{\bar Z}(-2\bar E)) } = (-1)^{n-1}\sum_{j\ne
0}\frac{n_j}{j} ,
\eeq
where $n_j$ is the dimension of the weight
$j$ part $V_j$ of the Zariski tangent space of $X$ at $p$ and
$n=\sum_{j\ne 0}n_j$. 

For this purpose, we may assume $X^T$ is a
point $p$, and $Y$ be so that $N_{Y^T/Y}|_p$ is the moving part of $T_p X$.
Let $V=\oplus_{j\ne 0}V_j$ with $p$ identified with $0$. 
Then
\[
e(\bar\pi_{(j)}\sta \cN_{(-j)}\dual(1))\cap 
[\PP\cN_{(j)}\times_{Y^T} p]=e(V_{-j}^\vee (1))\cap [\PP V_j]
\]
on $\PP V$ and \eqref{eq-ptwcf} is equivalent to
\[
\sum_{j\ne 0}\int_{\PP V_j} \frac{t\, e(V_{-j}^\vee (1))}{e(N\virt|_{\PP
V_j})c_1(\cO_{\bar Z}(-2\bar E)|_{\PP V_j})} =
(-1)^{n-1}\sum_{j\ne 0}\frac{n_j}{j} .
\]
Let $\zeta\in A^1(\PP V_j)$ be the generator of $A^*(\PP V_j)$
satisfying $\int_{\PP V_j} \zeta^{n_j-1}=1$. Then $$c_1(\cO_{\bar
Z}(-2\bar E)|_{\PP V_j})=2(-jt+\zeta).$$ Furthermore, we have
\[
e(\pi^*\Omega_Y(-E)|_{\PP V_j})=\prod_{i\ne
0}(-it-jt+\zeta)^{n_i},\quad e(V_{-j}^\vee (1)|_{\PP
V_j})=\zeta^{n_{-j}}.
\]
The normal bundle of $\PP V_j$ has Euler class
\[
(jt-\zeta)\prod_{i\ne j,0}((i-j)t+\zeta)^{n_i}
\]
and thus the Euler class of $N\virt$ is
\[
e(N\virt|_{\PP V_j})=(jt-\zeta) \prod_{i\ne j,
0}((i-j)t+\zeta)^{n_i}/\prod_{i\ne -j,0}(-it-jt+\zeta)^{n_i}.
\]
Therefore we have
\[
\sum_j \int_{\PP V_j}\frac{t\, e(V_{-j}^\vee (1))}{e(N\virt|_{\PP
V_j})c_1(\cO_{\bar Z}(-2\bar E)|_{\PP V_j})}
=\qquad\qquad\qquad\qquad\qquad\qquad
\]
\[\qquad\qquad\qquad=
\sum_j\int_{\PP V_{j}} \frac{t\prod_{i\ne
0}(-it-jt+\zeta)^{n_i}}{2(-jt+\zeta)(jt-\zeta)\prod_{i\ne j,
0}((i-j)t+\zeta)^{n_i}}
\]
\[
\qquad=-\sum_j \int_{\PP V_{j}}\frac{ t^{n_{j}-1}\prod_{i\ne 0
}(-i-j+\frac{\zeta}{t})^{n_i}}{
2(j-\frac{\zeta}{t})^{2}\prod_{i\ne
0,j}(i-j+\frac{\zeta}{t})^{n_i}}
\]
\[
\qquad\quad=-\sum_j \res_{x=0}\frac{x^{-n_{j}}\prod_{i\ne 0
}(-i-j+x)^{n_i}}{2(j-x)^2\prod_{i\ne 0,j}(i-j+x)^{n_i}}
\]
\[
\qquad\quad=-\sum_j \res_{x=0}\prod_{i\ne 0}\left(\frac{-i-j+x}{i-j+x}
\right)^{n_i}\frac{1}{2(j-x)^2}
\]
\[
\!\!= -\sum_j \res_{z=-j}\prod_{i\ne 0}\left(\frac{-i+z}{i+z}
\right)^{n_i}\frac{1}{2z^2}.
\]
%
\[
\qquad\qquad\qquad\qquad\ =\res_{z=0}\prod_{i\ne 0}\left(\frac{-i+z}{i+z}
\right)^{n_i}\frac{1}{2z^2}
=\frac12 \frac{d}{dz}|_{z=0}\prod_{i\ne 0} \left(\frac{-i+z}{i+z}
\right)^{n_i}
\]
\[ \!\!\!\!\!\!\!\!\!\!\!\!\!\!\!\!\!\!\!\!\!\!\!\!\!\!\!\!\!\!\!\!\!\!\!\!\!\!\!\!\!\!\!\!\!\!\!\!\!\!=(-1)^{n-1}\sum_{j\ne 0}\frac{n_{j}}{j}.
\]
This proves the theorem.
\end{proof}

For example, if there are only two weight spaces of weights $1$
and $-1$ respectively, then the wall crossing is
\[
(-1)^{n_++n_--1}\cdot (n_+-n_-)\cdot \deg [X^T]\virt,
\]
where $n_+$ and $n_-$ are the dimensions of the positive and
negative weight spaces respectively, of the moving part of the
Zariski tangent space.

\vsp

In the situation of simple wall crossing (Definition \ref{assum1})
in $D^b(\mathrm{Coh} S)$ for $S$ a Calabi-Yau 3-fold, we obtain
the wall crossing formula formulated in Corollary \ref{coro1.3}.


\begin{proof}[Proof of Corollary \ref{coro1.3}] We use the notation of \S\ref{subsection2.3}.
The automorphism group of a point in $X$ lying over $E_1\oplus
E_2$ with $E_1\in M_1$ and $E_2\in M_2$ is $\ZZ_\nu$ with $\nu$ in
\eqref{defnu}. Hence $$[X^T]=\frac1\nu [M_1]\virt\times
[M_2]\virt.$$ From the definition of the $\CC^*$-action on $X$,
the nontrivial weights on the Zariski tangent space are
$\frac1\nu$ on $Ext^1(E_2,E_1)$ and $-\frac1\nu$ on
$Ext^1(E_1,E_2)$ respectively. In the formula of Theorem
\ref{theo-sym}, the weights $\pm \frac1\nu$ that go in the
denominator cancel out the size of the automorphism group $\nu$.
Thus we obtain the corollary.
\end{proof}

\begin{rema}
In a subsequent paper, we will generalize our wall crossing
formula to the case $\cM_{\pm}\subset [X/G]$ where $G$ is any
complex reductive group acting on the semistable part $X$ of a
projective scheme. \end{rema}

\begin{appendix}
\section{Simple wall crossing for perfect obstruction theories}
In this section, we prove a virtual analogue of a wall crossing
formula for $\CC^*$-flips.

Since ${\CC\sta}$ acts on $X_\pm$ with finite stabilizers, a
${\CC\sta}$-equivariant perfect obstruction theory of $X$ induces a
perfect obstruction theory of $M_\pm$. Let $[M_\pm]\virt$ be the
associated virtual fundamental cycle \cite{BF, LT}. 


\begin{defi} Suppose $M_\pm\subset \cM$ are proper simple flips.
Let $d$ be their virtual dimension and $\alpha\in A^d_{\CC^*}(X)$ be
an equivariant cohomology class.
Then the
\emph{wall crossing term} of $\alpha$ is defined as
$$\delta[\alpha]=\alpha_+\cdot[M_+]\virt-\alpha_-\cdot[M_-]\virt,
$$
where $\alpha_\pm\in A^d(M_\pm)$ are the classes induced by
restricting $\alpha$ to $X_\pm$ and applying the isomorphisms
$A^*_{{\CC\sta}}(X_\pm)\cong A^*(M_\pm)$.
\end{defi}





\begin{theo}\label{thm-1}
Let $M_\pm$ be simple flips in the quotient stack
$\cM=[X/\CC^*]$ that has a perfect obstruction theory.
Suppose $X$ embeds $\csta$-equivariantly into a smooth DM-stack.
Then for $\alpha\in A^d_{\CC^*}(X)$, the wall crossing is given by
$$\delta[\alpha]=\sum_i \mathrm{res}
\Bigl(\alpha\cdot\frac{[X_i]\virt}{e(N_i\virt)}\Bigr)\in \QQ,
$$
where $X_i$ are connected components of $X^{\CC\sta}$ and $N_i\virt$
are the $\CC\sta$-equivariant virtual normal bundles of $X_i$ in
$X$.
\end{theo}

Here the residue operator $\mathrm{res}=\mathrm{res}_{t=0}$ is taken
after expanding the right hand side as a Laurent series in $t$,
where $t$ is the generator of $A^*_{\CC^*}(pt)$.


Let $t\in A^1_{\CC^*}(pt)$ be the generator of the $H$-equivariant Chow
ring $A^*_{\CC^*}(pt)$.  Let $\alpha\in A^d_{T}(X)$. Let
$\alpha_\pm\in A^d(M_\pm)$ be the classes induced by restricting
$\alpha$ to $X_\pm$ and applying the isomorphisms
$A^*_{T}(X_\pm)\cong A^*(M_\pm).$ Let $\tilde{\alpha}\in A^d(Z)$ be
the class induced by pulling back $\alpha$ to $(X\times \PP^1)^s$
and applying the isomorphism
$$A^*_{T}((X\times\Po)^s)\cong
A^*(Z).
$$
Let $\alpha_i$ be the restriction of $\alpha$ to the fixed point
component $X_i$. Then the restrictions of $t\tilde{\alpha}$ to the
fixed point set $M_\pm$ and $X_i$ coincide with $t\alpha_\pm$ and
$t\alpha_i$.

We apply the virtual localization theorem \cite{GrPand} to obtain
the following.

\begin{lemm} Let $X$ be a \DM stack equipped with a
$\CC^*$-equivariant perfect obstruction theory and  $X$ embeds
$\csta$-equivariantly into a smooth DM-stack. Then we have
\[
[Z]\virt=\imath_*\sum \frac{[X_i]\virt}{e(N_i\virt)} +
[M_+]\virt/(-t)+[M_-]\virt/t.
\]
\end{lemm}

Here $X_i$ are the connected components of the $T$-fixed locus $X^T$
of $X$, and $N_i\virt$ denotes the virtual normal bundle of $X_i$ in
$X$. Note that if $X$ satisfies the virtual localization
requirement, so does $Z$ by construction.

We pair the homology class $[Z]\virt$ with $t\tilde{\alpha}\in
A^{d+1}_H(Z)$ and then take the degree zero part in $t$. Since
$[Z]\virt\in A^H_{d+1}(Z)$ has only terms of nonnegative degrees
in $t$,
$$\deg \bl t\tilde\alpha\cdot[Z]\virt\br =0.
$$
Therefore, upon moving the terms for $M_\pm$ to the left hand side,
we obtain the desired wall crossing formula:
$$\delta[\alpha]=\sum_i \mathrm{res}\Bigl(
\alpha\cdot\frac{[X_i]\virt}{e(N_i\virt)}\Bigr)\in \QQ.
$$
This completes the proof of Theorem \ref{thm-1}.
\end{appendix}


\bibliographystyle{amsplain}

\end{document}